\documentclass{amsart}

\usepackage{amsmath} %
\usepackage{graphicx} %
\usepackage{hyperref} %
\usepackage{lipsum} %
\usepackage{xcolor}
\usepackage{soul}
\usepackage{yhmath}
\usepackage{upgreek}

\usepackage{pgf,tikz,pgfplots}
\pgfplotsset{compat=1.15}
\usepackage{mathrsfs}
\usetikzlibrary{arrows}

\newtheorem{theorem}{Theorem}[section]
\newtheorem{proposition}[theorem]{Proposition}
\newtheorem{corollary}[theorem]{Corollary}
\newtheorem{lemma}[theorem]{Lemma}
\theoremstyle{definition}
\newtheorem{remark}[theorem]{Remark}
\newtheorem{definition}[theorem]{Definition}

\DeclareMathOperator{\Real}{Re}
\DeclareMathOperator{\Imaginary}{Im}
\renewcommand{\Re}{\Real}
\renewcommand{\Im}{\Imaginary}

\newcommand{\Lie}{\mathcal{L}}
\newcommand{\vol}{\mathrm{vol}}
\renewcommand{\epsilon}{\varepsilon}

\title[Discrete Laplacians]{Discrete Laplacians -- spherical and hyperbolic}
\author[I.Izmestiev]{Ivan Izmestiev}
\address{TU Wien,Wiedner Hauptstraße 8-10/104, A-1040 Wien, Austria}
\email{izmestiev@dmg.tuwien.ac.at}
\author[W.Lam]{Wai Yeung Lam}
\address{Department of Mathematics, University of Luxembourg, Maison du nombre, 6 avenue de la Fonte, L-4364 Esch-sur-Alzette, Luxembourg}
\email{wyeunglam@gmail.com}
\thanks{
This research was funded in part by the Austrian Science Fund (FWF) 10.55776/F77. The second author was partially supported by the FNR grant CoSH O20/14766753. For open access purposes, the authors have applied a CC BY public copyright license to any authors accepted manuscript version arising from this submission. This is the Author Accepted Manuscript version of an article published in final form at \href{https://doi.org/10.1112/jlms.70235}{https://doi.org/10.1112/jlms.70235}.
}


\begin{document}

\begin{abstract}
The discrete Laplacian on Euclidean triangulated surfaces is a well-established notion.
We introduce discrete Laplacians on spherical and hyperbolic triangulated surfaces.
On the one hand, our definitions are close to the Euclidean one in that the edge weights contain the cotangents of certain combinations of angles and are non-negative if and only if the triangulation is Delaunay.
On the other hand, these discretizations are structure-preserving in several respects.
We prove that the area of a convex polyhedron can be written in terms of the discrete spherical Laplacian of the support function, whose expression is the same as the area of a smooth convex body in terms of the usual spherical Laplacian.
We show that the conformal factors of discrete conformal vector fields on a triangulated surface of curvature $k \in \{-1,1\}$ are $-2k$-eigenfunctions of our discrete Laplacians, exactly as in the smooth setting.
The discrete conformality can be understood here both in the sense of the vertex scaling and in the sense of circle patterns.
Finally, we connect the $-2k$-eigenfunctions to infinitesimal isometric deformations of a polyhedron inscribed into corresponding quadrics.
\end{abstract}

\maketitle

\section{Introduction}

Let $G=(V,E)$ be a graph equipped with vertex weights $d \colon V\to \mathbb{R} \setminus \{0\}$ and edge weights $c \colon E \to \mathbb{R}$.
Denote $d(i)$ by $d_i$ and $c(\{i,j\})$ by $c_{ij}=c_{ji}$.
The (normalized) \emph{discrete Laplacian} is a linear operator $\triangle \colon \mathbb{R}^{V}\to\mathbb{R}^{V}$ given by the formula
\begin{equation}
\label{eqn:DiscLapl}
(\triangle u)_i := \frac{1}{d_i} \sum_{j} c_{ij}(u_j-u_i) \quad \forall i \in V,
\end{equation}
where the sum ranges over all vertices adjacent to the vertex $i$ but can be extended to all $j \in V$ by setting $c_{ij} = 0$ for $ij \notin E$.
Vectors $u \in \mathbb{R}^V$ can be naturally viewed as functions $u:V \to \mathbb{R}$.
A function $u$ is called (discrete) \emph{harmonic} if $\triangle u = 0$.

The linear operator \eqref{eqn:DiscLapl} is self-adjoint with respect to the following inner product on $\mathbb{R}^V$:
\[
\langle u, v \rangle_d := \sum_i d_i u_i v_i.
\]
Observe that
\begin{equation}
\label{eqn:LaplWeak}
\langle \triangle u, v \rangle_d = \sum_i v_i \sum_j c_{ij}(u_j - u_i) =
-\sum_{\{i,j\}} c_{ij} (u_j - u_i)(v_j - v_i),
\end{equation}
so that the Laplacian ``in the weak sense'' is independent of the choice of the vertex weights.
The choice of vertex weights also has no influence on the space of harmonic functions, but it does affect the non-zero eigenvalues and the corresponding eigenspaces.
The above formula shows that, if the weights $c_{ij}$ are non-negative, then $\triangle$ is negative semidefinite.
If all edge weights are positive, then harmonic functions are those which are constant on each connected component of the graph.

Setting $c_{ij} = 1$ whenever $ij \in E$ and $d_i = 1$ one obtains the combinatorial Laplacian, which is related to combinatorial properties of the graph such as the number of spanning trees and the expansion.
Discrete Laplacians with arbitrary weights were studied in \cite{ColinDeVerdiere1998}. The basics of the theory including the definition of a discrete Hodge Laplacian on simplicial cochains were laid down in \cite{Eckmann1945}.

When the graph carries some geometric data, one can ask for a natural choice of weights reflecting those data.
If $(V,E)$ is the graph of a triangulated surface where each triangle is equipped with a Euclidean metric, then there is a well known assignment
\begin{align}\label{eq:eweight}
	c_{ij} := \frac{\cot \alpha^{k}_{ij} + \cot \alpha^{l}_{ij}}{2},
\end{align}
where $\alpha_{ij}^k$ and $\alpha_{ij}^l$ are the angles opposite to the edge $ij$, see Figure \ref{fig:TwoTriangles}, left.
This discrete Euclidean Laplacian, going now under the name ``the cotangent Laplacian'', was discovered independently in various contexts.
It appears in \cite{Duffin1959} related to electrical networks, in \cite{ColinDeVerdiere1998} as an example with a reference to \cite{Ciarlet1978}, and in \cite{Pinkall1993} that motivates the definition of discrete minimal surfaces.
Later on, relations were discovered to statistical mechanics \cite{Kenyon2002,Kenyon2018,Smirnov2010}, the variation of volume of hyperbolic polyhedra \cite{Bobenko2010}, deformations of circle patterns \cite{Lam2015a}, and isometric deformations of polyhedra \cite{Lam2015}.
The cotangent Laplacian is also widely used in geometry processing.

\begin{figure}
\begin{center}
\begin{picture}(0,0)%
\includegraphics{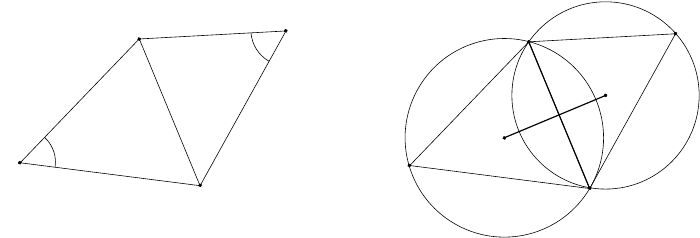}%
\end{picture}%
\setlength{\unitlength}{2486sp}%
\begingroup\makeatletter\ifx\SetFigFont\undefined%
\gdef\SetFigFont#1#2#3#4#5{%
  \reset@font\fontsize{#1}{#2pt}%
  \fontfamily{#3}\fontseries{#4}\fontshape{#5}%
  \selectfont}%
\fi\endgroup%
\begin{picture}(8887,3002)(3733,-1432)
\put(11227,-1012){\makebox(0,0)[lb]{\smash{{\SetFigFont{8}{9.6}{\rmdefault}{\mddefault}{\updefault}{\color[rgb]{0,0,0}$p_i$}%
}}}}
\put(9849,-374){\makebox(0,0)[lb]{\smash{{\SetFigFont{8}{9.6}{\rmdefault}{\mddefault}{\updefault}{\color[rgb]{0,0,0}$o_{ijk}$}%
}}}}
\put(10242,1169){\makebox(0,0)[lb]{\smash{{\SetFigFont{8}{9.6}{\rmdefault}{\mddefault}{\updefault}{\color[rgb]{0,0,0}$p_j$}%
}}}}
\put(6350,-880){\makebox(0,0)[lb]{\smash{{\SetFigFont{8}{9.6}{\rmdefault}{\mddefault}{\updefault}{\color[rgb]{0,0,0}$p_i$}%
}}}}
\put(5256,1193){\makebox(0,0)[lb]{\smash{{\SetFigFont{8}{9.6}{\rmdefault}{\mddefault}{\updefault}{\color[rgb]{0,0,0}$p_j$}%
}}}}
\put(3748,-677){\makebox(0,0)[lb]{\smash{{\SetFigFont{8}{9.6}{\rmdefault}{\mddefault}{\updefault}{\color[rgb]{0,0,0}$p_k$}%
}}}}
\put(4498,-354){\makebox(0,0)[lb]{\smash{{\SetFigFont{8}{9.6}{\rmdefault}{\mddefault}{\updefault}{\color[rgb]{0,0,0}$\alpha_{ij}^k$}%
}}}}
\put(6577,838){\makebox(0,0)[lb]{\smash{{\SetFigFont{8}{9.6}{\rmdefault}{\mddefault}{\updefault}{\color[rgb]{0,0,0}$\alpha_{ij}^l$}%
}}}}
\put(8608,-686){\makebox(0,0)[lb]{\smash{{\SetFigFont{8}{9.6}{\rmdefault}{\mddefault}{\updefault}{\color[rgb]{0,0,0}$p_k$}%
}}}}
\put(11498,323){\makebox(0,0)[lb]{\smash{{\SetFigFont{8}{9.6}{\rmdefault}{\mddefault}{\updefault}{\color[rgb]{0,0,0}$o_{ijl}$}%
}}}}
\put(12380,1230){\makebox(0,0)[lb]{\smash{{\SetFigFont{8}{9.6}{\rmdefault}{\mddefault}{\updefault}{\color[rgb]{0,0,0}$p_l$}%
}}}}
\put(7415,1269){\makebox(0,0)[lb]{\smash{{\SetFigFont{8}{9.6}{\rmdefault}{\mddefault}{\updefault}{\color[rgb]{0,0,0}$p_l$}%
}}}}
\end{picture}%
\end{center}
\caption{Two adjacent triangles and their circumcircles.}
\label{fig:TwoTriangles}
\end{figure}
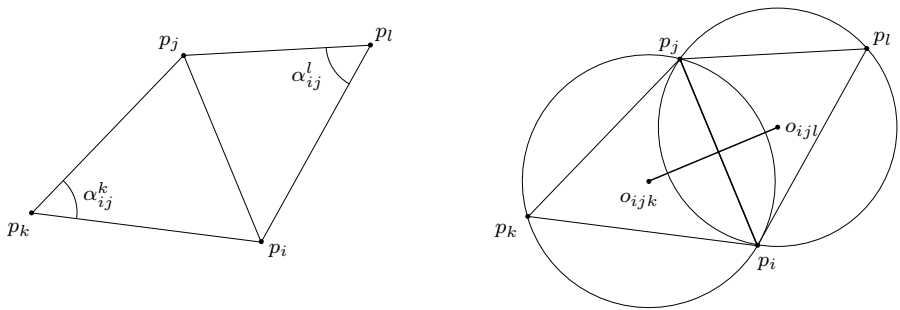

If we assume that the triangle $p_ip_jp_k$ is positively oriented, then the edge weight \eqref{eq:eweight} satisfies the equation
\begin{equation}
\label{eqn:CotWeightGeom}
o_{ijk} - o_{ijl} = c_{ij}\, J(p_j - p_i),
\end{equation}
where $o_{ijk}$ and $o_{ijl}$ are the circumcenters of the triangles $p_ip_jp_k$ and $p_ip_jp_l$, and $J$ is rotation by $90^\circ$ in the positive direction, see Figure \ref{fig:TwoTriangles}.
This implies that the following three conditions are equivalent:
\begin{enumerate}
\item
The weight $c_{ij}$ is nonnegative.
\item
The so-called local Delaunay condition: vertices $p_k$ and $p_l$ do not lie inside the circumcircles of $p_ip_jp_l$ and $p_ip_jp_k$, respectively.
\item
$\alpha_{ij}^k + \alpha_{ij}^l \le \pi$
\end{enumerate}
If $\alpha_{ij}^k + \alpha_{ij}^l = \pi$, then the circumcircles of triangles $ijk$ and $ijl$ coincide, and the edge weight $c_{ij}$ vanishes.
If all $c_{ij}$ are non-negative, then formula \eqref{eqn:LaplWeak} implies that the symmetric bilinear form $(u,v) \mapsto \langle \triangle u, v \rangle_d$ is negative semidefinite.

The vertex weights of the cotangent Laplacian are sometimes set to $1$, and sometimes chosen to be equal to the areas of the dual faces.
Here the dual face corresponding to the vertex $p_i$ is the polygon with the vertices at the circumcenters of triangles adjacent to $p_i$.
If the triangulation is Delaunay, then the dual faces are the so called Voronoi cells.
For a recent survey on the cotangent Laplacian and a generalization thereof see \cite{Glickenstein2024}.

In this article we propose discrete Laplacians adapted to the spherical and the hyperbolic metrics.
\begin{definition}
\label{def:dissph}
Let $(V,E)$ be the graph of a triangulated surface where each triangle is equipped with a spherical metric.
The \emph{discrete spherical Laplacian} $\triangle_s \colon \mathbb{R}^V \to \mathbb{R}^V$ is defined as
\[
(\triangle_s u)_i =   \frac{1}{d_i} \sum_j c_{ij} (u_j - u_i)
\]
with the edge weights and the vertex weights given by
\[
c_{ij}:=\frac{\tan\left( \dfrac{\alpha_{jk}^i + \alpha_{ki}^j - \alpha_{ij}^k}{2}\right) + \tan\left( \dfrac{\alpha_{lj}^i + \alpha_{il}^j - \alpha_{ji}^l}{2}\right)}{2\cos^2 \dfrac{\lambda_{ij}}{2}},
\]
\[
d_i := \sum_j c_{ij} \sin^2 \frac{\lambda_{ij}}{2},
\]
where $\lambda_{ij}$ is the length of the edge $ij$, and $\alpha_{ij}^k$ is the angle  at the vertex $k$ in the spherical triangle $ijk$.

\end{definition}

\begin{definition}\label{def:dishyp}
Let $(V,E)$ be the graph of a triangulated surface where each triangle is equipped with a hyperbolic metric.
The \emph{discrete hyperbolic Laplacian} $\triangle_h \colon \mathbb{R}^V \to \mathbb{R}^V$ is defined as
\[
(\triangle_h u)_i =   \frac{1}{d_i} \sum_j c_{ij} (u_j - u_i)
\]
with the edge weights and the vertex weights given by
\[
c_{ij}:=\frac{\tan\left( \dfrac{\alpha_{jk}^i + \alpha_{ki}^j - \alpha_{ij}^k}{2}\right) + \tan\left( \dfrac{\alpha_{lj}^i + \alpha_{il}^j - \alpha_{ji}^l}{2}\right)}{2\cosh^2 \dfrac{\lambda_{ij}}{2}},
\] 
\[
d_i =  \sum_j c_{ij} \sinh^2 \frac{\lambda_{ij}}{2},
\]
where $\lambda_{ij}$ is the length of the edge $ij$, and $\alpha_{ij}^k$ is the angle at the vertex $k$ in the hyperbolic triangle $ijk$.
\end{definition}

The edge weights of the discrete hyperbolic Laplacian have appeared implicitly in \cite{Bobenko2010,Glickenstein2017,Leibon2002}.

We will show that both in the spherical and in the hyperbolic cases the edge weights $c_{ij}$ are non-negative if and only if the local Delaunay condition is satisfied: the circumcircles of triangles do not contain the opposite vertices of the adjacent triangles.
Besides, we will relate the vertex weights to the areas of dual faces, whose definition is different from the Euclidean case.
In the limit when spherical or hyperbolic triangles become small, the angle sum $\alpha_{ij}^k + \alpha_{jk}^i + \alpha_{ki}^j$ tends to $\pi$ and one recovers the Euclidean edge weights \eqref{eq:eweight}.
In the same limit, the non-Euclidean dual faces tend to the Euclidean ones, so the same continuity takes place for the vertex weights.

Our definition of discrete spherical and hyperbolic Laplacians is justified by discrete analogs of several theorems about the Laplace--Beltrami operators on surfaces of constant Gaussian curvature.
Let us first state these theorems.

\begin{theorem}
\label{thm:AreaSmooth}
Let $P \subset \mathbb{R}^3$ be a convex body with smooth boundary, and let $h \colon \mathbb{S}^2 \to \mathbb{R}$ be its support function.
Then the surface area of $P$ is equal to
\[
\frac12 \int_{\mathbb{S}^2} h (2h + \triangle_g h)\, d\vol_g.
\]
Here $\triangle_g$ and $d\vol_g$ are the Laplace--Beltrami operator and the area element of the unit sphere.
\end{theorem}

The support function of a convex body $P \subset \mathbb{R}^3$ is the map $h \colon \mathbb{S}^2 \to \mathbb{R}$ that associates to every $v \in \mathbb{S}^2$ the signed distance from $0$ to the tangent plane of $P$ at the point with the outward unit normal $v$.

The formula of Theorem \ref{thm:AreaSmooth} is due to Minkowski \cite{Minkowski1903}, see also \cite{Hilbert1910}. The next theorem follows from the curvature formula under change of conformal metrics.

\begin{theorem}
\label{thm:ConfFieldSmooth}
Let $(M, g)$ be a Riemannian surface of constant Gaussian curvature $k$, and let $\xi$ be a conformal vector field on $M$ with the conformal factor $f$:
\[
\Lie_\xi g = fg,
\]
where $\Lie$ is the Lie derivative.
Then the function $f$ is an eigenfunction of the Laplace--Beltrami operator of the metric $g$, namely one has
\[
\triangle_g f = -2k f.
\]
\end{theorem}

The eigenfunctions are also related to infinitesimal isometric deformations in space. An infinitesimal isometric deformation of a surface in $\mathbb{R}^3$ (or, more generally, of any smooth submanifold of a Riemannian manifold) is a vector field such that any deformation of the surface with the initial velocity equal to this field preserves the lengths of curves on the surface in the first order.

\begin{theorem}
\label{thm:SphIIDSmooth}
Let $\Omega$ be an open subset of $\mathbb{S}^2$, and let $q:\Omega \to \mathbb{R}^3$ be an infinitesimal isometric deformation of $\Omega$.
Decompose $q$ into the radial and the tangential components:
\begin{equation}
\label{eqn:qDecomp}
q = f \nu + \xi,
\end{equation}
where $\xi$ is a tangent vector field to $\Omega$, and $\nu$ is the outward unit normal field of $\mathbb{S}^2$.
Then the radial component is an eigenfunction of the spherical Laplace--Beltrami operator with the eigenvalue $-2$:
\[
\triangle_g f = -2f.
\]
Besides, the vector field $\xi$ is conformal, with the conformal factor equal to minus twice the radial component:
\begin{equation}
\label{eqn:xiConf}
\Lie_\xi g = -2f g.
\end{equation}
\end{theorem}

For a hyperbolic analog of the above theorem let us recall the hyperboloid model of the hyperbolic plane.
The Minkowski inner product
\begin{equation}
\label{eqn:MinkProd}
\langle x, y \rangle_{2,1} = -x_0y_0 + x_1y_1 + x_2y_2
\end{equation}
restricts on the tangent planes of the (upper half of the) one-sheeted hyperboloid
\begin{equation}
\label{eqn:HypMink}
\mathbb{H}^2 = \{(x_0, x_1, x_2) \in \mathbb{R}^3 \mid x_0^2 - x_1^2 - x_2^2 = 1, x_0 > 0\}
\end{equation}
to a field of positive definite symmetric bilinear forms, and thus defines a Riemannian metric $g$ on $\mathbb{H}^2$.
This metric has constant Gaussian curvature $-1$, therefore its Laplace--Beltrami operator $\triangle_g$ is the hyperbolic Laplacian.

\begin{theorem}
\label{thm:HypIIDSmooth}
Let $\Omega$ be an open subset of the upper half of the one-sheeted hyperboloid \eqref{eqn:HypMink}, and let $q \colon \Omega \to \mathbb{R}^3$ be an infinitesimal isometric deformation of $\Omega$.
Then the function
\[
f \colon \Omega \to \mathbb{R}, \quad f(p) = \langle p, q(p) \rangle
\]
is an eigenfunction of the hyperbolic Laplace--Beltrami operator with the eigenvalue~$2$:
\[
\triangle_g f = 2f,
\]
where $g$ is the restriction of the Minkowski inner product \eqref{eqn:MinkProd} to the tangent planes of $\Omega$.
Besides, the vector field
\[
\xi(p) = \overline{q}(p) + \langle p, q(p) \rangle p,
\]
where $\overline q = c \circ q$ with $c(x_0, x_1, x_2) = (-x_0, x_1, x_2)$, is a conformal tangent vector field on $\Omega$, with the conformal factor equal twice the function $f$:
\begin{equation}
\label{eqn:LieHypDiscr}
\Lie_\xi g = -2fg.
\end{equation}
\end{theorem}

Theorems \ref{thm:SphIIDSmooth} and \ref{thm:HypIIDSmooth} are proved in \cite{IZ24}. Below are the discrete counterparts of the above theorems that we shall focus on. First, we have a discrete analogue of Theorem \ref{thm:AreaSmooth}.

\begin{theorem}
\label{thm:AreaDiscrete}
Let $P \subset \mathbb{R}^3$ be a convex polyhedron with outward unit normals $p_1, \ldots, p_n$ and the corresponding support numbers $h_1, \ldots, h_n$, that is
\[
P = \{x \in \mathbb{R}^3 \mid \langle x, p_i \rangle \le h_i \text{ for all }i\},
\]
so that none of the linear inequalities is redundant.
Then the surface area of $P$ is equal to
\[
\frac12 \langle h, 2h + \triangle_s h \rangle_d.
\]
Here $\langle u, v \rangle_d = \sum_i d_i u_i v_i$, and $\triangle_s$ is the discrete spherical Laplacian for the triangulation of $\mathbb{S}^2$ with the vertices $p_1, \ldots, p_n$ and the edges dual to the edges of $P$.
(If the dual subdivision is not a triangulation, subdivide its faces by any diagonals.)
\end{theorem}

We chose not to deal with hyperbolic analogs of Theorems \ref{thm:AreaSmooth} and \ref{thm:AreaDiscrete}, but they should hold with the sphere replaced by the hyperboloid in the Minkowski space.
The appropriate tools are provided by \cite{Fillastre2013}.

There are two different discrete conformality theories.
Both of them tell when two triangulated surfaces with the same combinatorics $(V, E, F)$ but different piecewise spherical, respectively hyperbolic, metrics are conformal.

One, proposed in the Euclidean case in \cite{Luo2004} and extended to the spherical and hyperbolic case in \cite{Bobenko2010,Gu2018}, defines discrete conformality by vertex scaling in terms of edge lengths $\lambda:E\to (0, +\infty)$. In the spherical case, two edge length functions $\lambda,\tilde{\lambda}$ are discrete conformal if there exists a function $u \colon V \to \mathbb{R}$ such that for all $ij \in E$
\[
\sin\frac{\widetilde{\lambda}_{ij}}2 = e^{\frac{u_i + u_j}2} \sin\frac{\lambda_{ij}}2
\]
while respectively in the hyperbolic case
\[
 \sinh\frac{\widetilde{\lambda}_{ij}}2 = e^{\frac{u_i + u_j}2} \sinh\frac{\lambda_{ij}}2.
\]
The analogy with the smooth case is clear: the pointwise scaling of a Riemannian metric is replaced by scaling of edge lengths with scaling factors assigned to vertices.
Observe that a discrete conformal deformation in this sense can change the total angles around the vertices.
This leads to the question what values of curvatures at the vertices can be realized within a given conformal class.
For an overview and a recent result in the spherical case see \cite{IzmestievProsanovWu2024}.

The other theory discretizes the angle preserving property of a conformal metric change.
For every edge $ij \in E$ consider the intersection angles of the neighboring circumcircles of the triangles adjacent to $ij$.
If they are the same for $\lambda$ and $\widetilde{\lambda}$, then these two metrics are declared conformally equivalent.
Note that according to this definition the conformality preserves the total angles around the vertices.
This definition generalizes the circle packing approach \cite{Stephenson2005}.

Although substantially different in general, the vertex scaling and intersection angle conformality notions are in bijection under infinitesimal deformations: the corresponding vector fields are related by a $\pi/2$ rotation.
Below is our discrete analog of Theorem \ref{thm:ConfFieldSmooth}.

\begin{theorem}
\label{thm:ConfFieldDiscreteSph}
Let $(V, E)$ be a geodesic triangulation of a simply-connected spherical surface $M$.
For a function $u \colon V \to \mathbb{R}$ the following two conditions are equivalent.
\begin{enumerate}
\item
$\triangle_s u = -2u$
\item
There is a tangent vector field $\xi \colon V \to TM$ such that when the vertices are moved with initial velocities $\xi$ the derivatives of the edge lengths $\lambda_{ij}$ satisfy
\[
\left( \log \sin \frac{\lambda_{ij}}2 \right)^{\boldsymbol{\cdot}} = \frac{u_i + u_j}2.
\]
\end{enumerate}
Besides, there is a natural bijection between such vector fields $\xi$ and tangent vector fields that infinitesimally preserve the intersection angles of adjacent circumcircles. The two are related by a rotation of angle $\frac{\pi}{2}$ in the tangent plane at each vertex.
\end{theorem}

\begin{theorem}
\label{thm:ConfFieldDiscreteHyp}
Let $(V, E)$ be a geodesic triangulation of a simply-connected hyperbolic surface $M$.
For a function $u \colon V \to \mathbb{R}$ the following two conditions are equivalent.
\begin{enumerate}
\item
$\triangle_h u = 2u$
\item
There is a tangent vector field $\xi \colon V \to TM$ such that when the vertices are moved with initial velocities $\xi$ the derivatives of the edge lengths $\lambda_{ij}$ satisfy
\[
\left( \log \sinh \frac{\lambda_{ij}}2 \right)^{\boldsymbol{\cdot}} = \frac{u_i + u_j}2.
\]
\end{enumerate}
Besides, there is a natural bijection between such vector fields $\xi$ and tangent vector fields that infinitesimally preserve the intersection angles of adjacent circumcircles. The two are related by a rotation of angle $\frac{\pi}{2}$ in the tangent plane at each vertex.
\end{theorem}

The corresponding statement for the Euclidean case (where $u$ must be a discrete harmonic function) is already known \cite{Glickenstein2016,Lam2015a}. Below is a discrete analog of 
Theorems \ref{thm:SphIIDSmooth}. 

\begin{theorem}
\label{thm:SphIIDDiscrete}
Let $P$ be a triangulated Euclidean polyhedral surface with vertices indexed by the set $V = \{i, j, \ldots\}$, all vertices $p_i$ lying on the unit sphere and such that the projection of $P$ to the sphere from its center is injective.
Denote by $P_s$ the triangulated spherical surface obtained by this projection.
Let $q \colon V \to \mathbb{R}^3$ be an infinitesimal isometric deformation of $P$.
Decompose $q$ into the radial and tangential components:
\[
q_i = f_i p_i + \xi_i, \quad f \colon V \to \mathbb{R}, \quad \xi \colon V \to \mathbb{R}^3, \quad \langle \xi_i, p_i \rangle = 0.
\]
Then the following holds:
\begin{enumerate}
\item
The radial component of $q$ is an eigenfunction of the discrete spherical Laplacian on $P_s$ with the eigenvalue $-2$:
\[
\triangle_s f = -2f.
\]
\item
Conversely, if $P$ is simply connected, then for every $-2$-eigenfunction $f$ of the discrete spherical Laplacian on $P_s$ there is an infinitesimal isometric deformation of $P$ with the radial component $f$.
\item
The vector field $\xi$ is an infinitesimal conformal deformation of $P_s$ with the conformal factor $-f$.
\end{enumerate}
\end{theorem}

The hyperbolic analog of Theorem \ref{thm:SphIIDDiscrete} holds with Euclidean space replaced by Minkowski space, which is stated in Theorem \ref{thm:HypIIDDiscretePrime}. We shall show that this leads to a discrete analog of 
Theorem \ref{thm:HypIIDSmooth}.

\begin{theorem}
\label{thm:HypIIDDiscrete}
Let $P$ be a triangulated Euclidean polyhedral surface with vertices indexed by the set $V = \{i, j, \ldots\}$, all vertices $p_i$ lying on the upper half of the one-sheeted hyperboloid \eqref{eqn:HypMink}
and such that the projection of $P$ to $\mathbb{H}^2$ from the origin is injective.
Denote by $P_h$ the triangulated hyperbolic surface obtained by this projection.
Let $q \colon V \to \mathbb{R}^3$ be an infinitesimal isometric deformation of $P$.
Consider the function
\[
f \colon V \to \mathbb{R}, \quad f_i = \langle p_i, q_i \rangle.
\]
Then the following holds:
\begin{enumerate}
\item
The function $f$ is an eigenfunction of the discrete hyperbolic Laplacian on $P_h$ with the eigenvalue $2$:
\[
\triangle_h f = 2f.
\]
\item
Conversely, if $P$ is simply connected, then for every $2$-eigenfunction $f$ of the discrete hyperbolic Laplacian on $P_h$ there is an infinitesimal isometric deformation of $P$ with the radial component $f$.
\item
The vector field
\[
\xi_i = \overline{q}_i + \langle p_i, q_i \rangle p_i,
\]
where $\overline{q} = c \circ q$, $c(x_0, x_1, x_2) = (-x_0, x_1, x_2)$, is an infinitesimal conformal deformation of $P_h$ with the conformal factor $f$.
\end{enumerate}
\end{theorem}

As applications, we prove the infinitesimal rigidity of circle patterns on closed hyperbolic surfaces (Theorem \ref{thm:disrigid} and Corollary \ref{cor:circlerigid}), which is related to the deformation space of circle patterns on surfaces with complex projective structures \cite{lam2024sym}. In  \cite{lam2024discrete}, the second named author established a connection between the discrete hyperbolic Laplacian and discrete harmonic maps between hyperbolic surfaces, which recovers surfaces from prescribed edge weights.

The rest of the article is split in two parts: Section \ref{sec:sphLap} deals with the spherical discrete Laplacian, while Section \ref{sec:HypLap} deals with its hyperbolic counterpart.
In Section \ref{sec:GeomMeaning} a geometric interpretation of the edge weights and vertex weights is given, relating them to the geometry of the circumscribed polyhedron, combinatorially dual to the given triangulation.
In Section \ref{sec:SurfArea} the area formula from Theorem \ref{thm:AreaDiscrete} is proved.
Certain results of that section are used in Section \ref{sec:InfDef} to establish a connection between the $-2$-eigenfunctions of the discrete spherical Laplacian and the radial components of infinitesimal deformations of inscribed polyhedra, which are the first two parts of Theorem \ref{thm:SphIIDDiscrete}.
Discrete conformality is discussed in Section \ref{sec:InfConfDef}, where Theorem \ref{thm:ConfFieldDiscreteSph} and the last part of Theorem \ref{thm:SphIIDDiscrete} are proved.
A similar order of exposition repeats in Section \ref{sec:HypLap}. In Section \ref{sec:circleinfrigid}, we apply the results to prove the infinitesimal rigidity of circle patterns on closed hyperbolic surfaces.

Theorem \ref{thm:ConfFieldDiscreteSph} and Theorem \ref{thm:ConfFieldDiscreteHyp} can be proved alternatively by the variation of angle formulas \cite{Glickenstein2011,Glickenstein2017}. Our discrete Laplacians can be extended to hyperbolic (or spherical) surfaces with conical singularities using intrinsic Delaunay triangulation as in the Euclidean case \cite{Bobenko2007}.
A weak Laplacian can be defined on a Sobolev space of functions on the sphere following \cite{DebinFillastre2022}, using the relation between the surface area and the Laplacian, see Theorem \ref{thm:AreaSmooth}.
It would be interesting to see whether our discrete spherical Laplacian can be related to this construction by extending the given values $u_i$ at a given finite subset of the sphere to a function of a certain kind.

\section{Discrete spherical Laplacian}\label{sec:sphLap}
\subsection{Geometric meaning of the edge weights and vertex weights}
\label{sec:GeomMeaning}
Let $(V, E, F)$ be a triangulated surface without boundary where every triangle is equipped with a spherical metric.
This defines a \emph{spherical cone-metric} on the surface: a geometric structure locally modelled on $\mathbb{S}^2$ and, in the neighborhoods of the vertices, on spherical cones.
For an edge $ij \in E$ let $k$ and $l$ be the other vertices of the triangles sharing that edge.
The vertices $i, j, k, l$ can be sent to points $p_i, p_j, p_k, p_l$ on $\mathbb{S}^2$ so that the union of the triangles $ijk$ and $ijl$ is mapped isometrically.
As we will be dealing with oriented angles, let us assume that the triangle $p_ip_jp_k$ is oriented positively and therefore the triangle $p_ip_jp_l$ is oriented negatively.

\begin{lemma}
\label{lem:SpherWeightsPos}
The following are equivalent:
\begin{enumerate}
\item
The point $p_l$ lies on or outside of the circumcircle of the spherical triangle $p_ip_jp_k$ (the local Delaunay condition).
\item
The dihedral angle at the line $p_ip_j$ between the halfplanes $p_ip_jp_k$ and $p_ip_jp_l$ is less than or equal $\pi$ when viewed from the center of the sphere.
\item
The weight $c_{ij}$ from Definition \ref{def:dissph} is nonnegative.
\end{enumerate}
Besides, the point $p_l$ is on the circumcircle if and only if the dihedral angle at $p_ip_j$ equals $\pi$ and if and only if $c_{ij} = 0$.
\end{lemma}
\begin{proof}
The circumcircle of $p_ip_jp_k$ is the intersection of the plane $p_ip_jp_k$ with $\mathbb{S}^2$, and its interior corresponds to the halfspace not containing the center of $\mathbb{S}^2$.
Therefore condition 1 is equivalent to the point $p_l$ and the center of $\mathbb{S}^2$ lying on the same side of the plane $p_ip_jp_k$ which, in turn, is equivalent to condition 2.

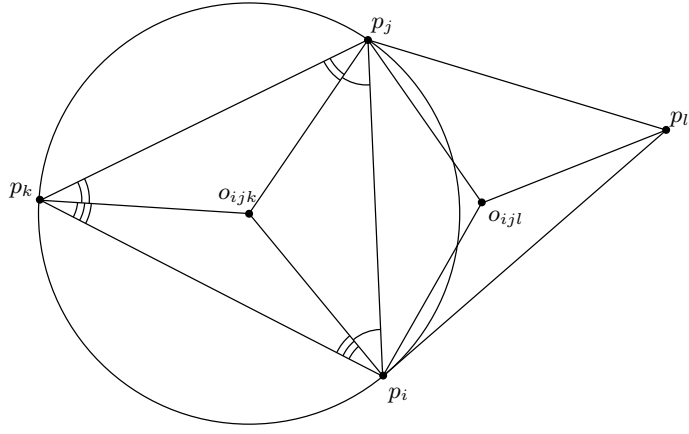
\begin{figure}[ht]
\begin{center}
\begin{picture}(0,0)%
\includegraphics{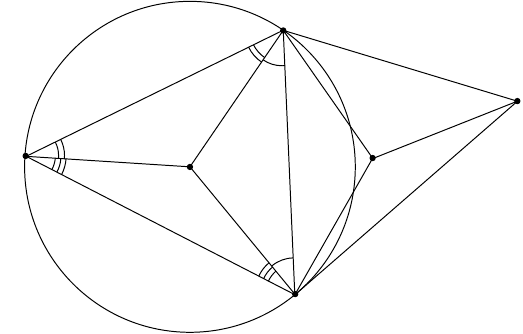}%
\end{picture}%
\setlength{\unitlength}{4144sp}%
\begingroup\makeatletter\ifx\SetFigFont\undefined%
\gdef\SetFigFont#1#2#3#4#5{%
  \reset@font\fontsize{#1}{#2pt}%
  \fontfamily{#3}\fontseries{#4}\fontshape{#5}%
  \selectfont}%
\fi\endgroup%
\begin{picture}(3982,2538)(619,-1860)
\put(1881,-503){\makebox(0,0)[lb]{\smash{{\SetFigFont{9}{10.8}{\rmdefault}{\mddefault}{\updefault}{\color[rgb]{0,0,0}$o_{ijk}$}%
}}}}
\put(2901,-1693){\makebox(0,0)[lb]{\smash{{\SetFigFont{9}{10.8}{\rmdefault}{\mddefault}{\updefault}{\color[rgb]{0,0,0}$p_i$}%
}}}}
\put(2799,518){\makebox(0,0)[lb]{\smash{{\SetFigFont{9}{10.8}{\rmdefault}{\mddefault}{\updefault}{\color[rgb]{0,0,0}$p_j$}%
}}}}
\put(3497,-635){\makebox(0,0)[lb]{\smash{{\SetFigFont{9}{10.8}{\rmdefault}{\mddefault}{\updefault}{\color[rgb]{0,0,0}$o_{ijl}$}%
}}}}
\put(634,-467){\makebox(0,0)[lb]{\smash{{\SetFigFont{9}{10.8}{\rmdefault}{\mddefault}{\updefault}{\color[rgb]{0,0,0}$p_k$}%
}}}}
\put(4586,-42){\makebox(0,0)[lb]{\smash{{\SetFigFont{9}{10.8}{\rmdefault}{\mddefault}{\updefault}{\color[rgb]{0,0,0}$p_l$}%
}}}}
\end{picture}%
\end{center}
\caption{Two adjacent spherical triangles and their circumcenters.}
\label{fig:SpherCircumcircle}
\end{figure}

Let $o_{ijk} \in \mathbb{S}^2$ be the center of the circumcircle of the triangle $p_ip_jp_k$.
For any three pairwise different and non-antipodal points $A, B, C \in \mathbb{S}^2$ denote by $\angle(\wideparen{AB}, \wideparen{AC})$ the oriented angle from the shortest great circle arc $\wideparen{AB}$ to the shortest great circle arc $\wideparen{AC}$.
Then, as the triangle $p_ip_jp_k$ is by our assumption positively oriented, one has
\[
\angle(\wideparen{p_ip_j}, \wideparen{p_io_{ijk}}) + \angle(\wideparen{p_io_{ijk}}, \wideparen{p_ip_k}) = \angle(\wideparen{p_ip_j}, \wideparen{p_ip_k}) = \alpha_{jk}^i.
\]
From this and the other two similar equations and from the equal pairs of angles in Figure \ref{fig:SpherCircumcircle} such as $\angle(\wideparen{p_ip_j}, \wideparen{p_io_{ijk}}) = \angle(\wideparen{p_jo_{ijk}}, \wideparen{p_jp_i})$ one computes
\[
\angle(\wideparen{p_ip_j}, \wideparen{p_io_{ijk}}) = \frac{\alpha_{ik}^j + \alpha_{jk}^i - \alpha_{ij}^k}2.
\]
As the triangle $p_ip_jp_l$ is negatively oriented, the corresponding formula for it has the opposite sign:
\[
\angle(\wideparen{p_ip_j}, \wideparen{p_io_{ijl}}) = - \frac{\alpha_{il}^j + \alpha_{jl}^i - \alpha_{ij}^l}2.
\]
The point $p_l$ lies outside of the circumcircle of the triangle $p_ip_jp_k$ if and only if the angle $\angle(\wideparen{p_io_{ijl}}, \wideparen{p_io_{ijk}})$ is positive, that is if and only if
\[
\frac{\alpha_{ik}^j + \alpha_{jk}^i - \alpha_{ij}^k}2 + \frac{\alpha_{il}^j + \alpha_{jl}^i - \alpha_{ij}^l}2 > 0.
\]
This is equivalent to
\[
\tan\left(\frac{\alpha_{ik}^j + \alpha_{jk}^i - \alpha_{ij}^k}2\right) + \tan\left(\frac{\alpha_{il}^j + \alpha_{jl}^i - \alpha_{ij}^l}2\right) > 0
\]
and hence to $c_{ij} > 0$ as $\tan \phi + \tan \psi = \frac{\sin(\phi + \psi)}{\cos\phi \cos\psi}$, and the angles $\phi$ and $\psi$ are in our case in the interval $\left(-\frac\pi2, \frac\pi2\right)$ being the (oriented) angles at the bases of isosceles triangles with legs of length $<\frac\pi2$.
This proves the equivalence of conditions 1 and 3.
The above arguments also prove the last statement of the lemma: $c_{ij} = 0$ is equivalent to the dihedral angle at $p_ip_j$ being equal to $\pi$ and to the circumcircle of $p_ip_jp_k$ containing $p_l$.
\end{proof}

Denote by $p_i^\ast, p_j^\ast, p_k^\ast, p_l^\ast$ the tangent planes to $\mathbb{S}^2$ at the points $p_i, p_j, p_k, p_l$ respectively.
Let $n_{ijk} \in \mathbb{R}^3$ be the intersection point of the planes $p_i^\ast$, $p_j^\ast$, $p_k^\ast$.
We call $n_{ijk}$ the \emph{dual vertex} corresponding to the face $ijk \in F$.
Connect the points $n_{ijk}$ and $n_{ijl}$ by a straight line segment, the \emph{dual edge} of the edge $ij \in E$.
As the points $n_{ijk}$ and $n_{ijl}$ lie on the intersection line of the planes $p_i^\ast$ and $p_j^\ast$, the vector $n_{ijk} - n_{ijl}$ is orthogonal to both unit vectors $p_i$ and $p_j$ and thus parallel to their cross product $p_i \times p_j$ (See Figure \ref{fig:3dtriangles}).

\begin{figure}[h!]
	 	\centering
	\includegraphics[width=0.65\textwidth]{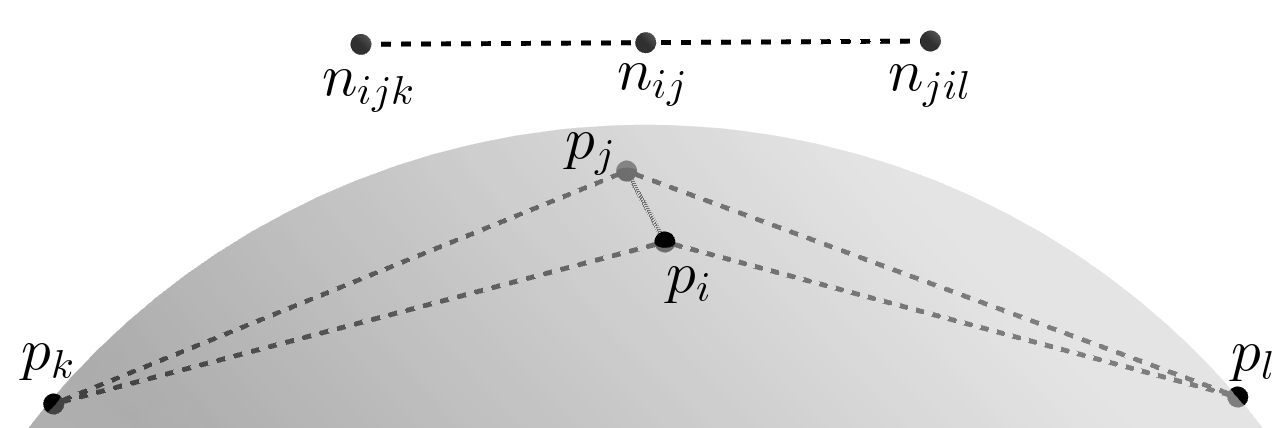}
	\caption{Two adjacent triangles in $\mathbb{R}^3$ with vertices $p_i,p_j,p_k,p_l$ on the sphere.}
	\label{fig:3dtriangles}
\end{figure}

If the total angle around the vertex $i$ is $2\pi$, that is $i$ is not a cone point of our spherical cone-metric, then all vertices adjacent to $i$ can be sent to $\mathbb{S}^2$ so that the union of triangles adjacent to $i$ is mapped isometrically.
The duals of the edges incident to $i$ form then a closed polygon in the plane $p_i^\ast$, the \emph{dual face} of the vertex $i$.
If the dihedral angles at all edges $p_ip_j$ are less or equal $\pi$, then the dual polygon of $i$ is convex, otherwise it is non-convex and self-intersecting.
Define the \emph{algebraic area} of the dual face as the sum of the areas of triangles $p_in_{ijk}n_{ijl}$ taken with the plus sign if $n_{ijk} - n_{ijl}$ is a positive multiple of $p_i \times p_j$ and with the minus sign otherwise.

\begin{lemma}
\label{lem:WeightsGeom}
For every edge $ij$ the dual edge is collinear with the cross product $p_i \times p_j$.
The scalar factor is the corresponding edge weight of the discrete spherical Laplacian:
\begin{equation}
\label{eqn:DualEdgeSpher}
n_{ijk} - n_{ijl} = c_{ij} (p_i \times p_j).
\end{equation}

In the non-singular case, when the union of all triangles adjacent to the vertex~$i$ is mapped isometrically to $\mathbb{S}^2$, the sum of the vertices adjacent to $p_i$ weighted with the corresponding edge weights is collinear with $p_i$.
The scalar factor is a combination of the edge and vertex weights:
\begin{equation}
\label{eqn:DualFaceSpher}
\sum_j c_{ij} p_j = \left( -2d_i + \sum_j c_{ij} \right) p_i.
\end{equation}
Moreover, $d_i$ is the algebraic area of the polygon dual to the vertex $i$.
\end{lemma}
\begin{proof}
Points on the line $p_i^\ast \cap p_j^\ast$ are equidistant from $p_i$ and $p_j$.
Hence the point $n_{ijk}$ is equidistant from $p_i, p_j$, and $p_k$.
It follows that the sphere center $o$, the triangle circumcenter $o_{ijk}$, and the point $n_{ijk}$ are collinear.
Thus the line $p_in_{ijk}$ is tangent to the great circle arc $\wideparen{p_io_{ijk}}$.
Denote by $n_{ij}$ the common foot of the perpendiculars from $p_i$ and $p_j$ on $p_i^\ast \cap p_j^\ast$, see Figure \ref{fig:DualPolyh}.
Then the line $p_in_{ij}$ is tangent to the arc $\wideparen {p_ip_j}$, and one has
\[
\angle(p_in_{ij}, p_in_{ijk}) = \angle(\wideparen{p_ip_j}, \wideparen{p_io_{ijk}}) = \frac{\alpha_{ik}^j + \alpha_{jk}^i - \alpha_{ij}^k}2.
\]
Consider the unit vector
\[
e_{ij} = \frac{p_i \times p_j}{\|p_i \times p_j\|}.
\]
The vector $n_{ijk} - n_{ij}$ is a positive multiple of $e_{ij}$ if and only if $\frac{\alpha_{ik}^j + \alpha_{jk}^i - \alpha_{ij}^k}2 > 0$.
From this and from the right-angled triangle $p_in_{ij}n_{ijk}$ one computes
\[
n_{ijk} - n_{ij} = \|n_{ij} - p_i\| \tan \frac{\alpha_{ik}^j + \alpha_{jk}^i - \alpha_{ij}^k}2 e_{ij}.
\]
The spherical distance $\lambda_{ij}$ between the points $p_i, p_j \in \mathbb{S}^2$ is the angle between the vectors $p_i, p_j \in \mathbb{R}^3$.
One has
\[
\|n_{ij} - p_i\| = \tan \frac{\lambda_{ij}}{2}, \quad
\|p_i \times p_j\| = \sin \lambda_{ij}.
\]
Substituting this into the above formula one obtains
\[
n_{ijk} - n_{ij} = \frac{\tan \dfrac{\alpha_{ik}^j + \alpha_{jk}^i - \alpha_{ij}^k}2}{2 \cos^2 \dfrac{\lambda_{ij}}{2}} \, p_i \times p_j.
\]
A similar formula holds for the vector $n_{ij} - n_{ijl}$.
Adding them together one obtains~\eqref{eqn:DualEdgeSpher}.

\begin{figure}
\begin{center}
\begin{picture}(0,0)%
\includegraphics{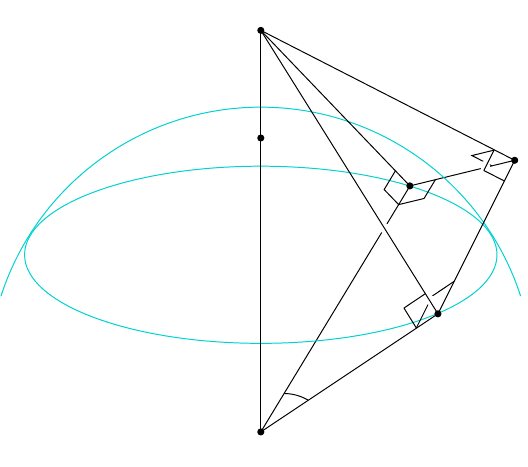}%
\end{picture}%
\setlength{\unitlength}{4144sp}%
\begingroup\makeatletter\ifx\SetFigFont\undefined%
\gdef\SetFigFont#1#2#3#4#5{%
  \reset@font\fontsize{#1}{#2pt}%
  \fontfamily{#3}\fontseries{#4}\fontshape{#5}%
  \selectfont}%
\fi\endgroup%
\begin{picture}(4014,3543)(-187,-1676)
\put(2079,-1079){\makebox(0,0)[lb]{\smash{{\SetFigFont{9}{10.8}{\rmdefault}{\mddefault}{\updefault}{\color[rgb]{0,0,0}$\lambda_{ij}$}%
}}}}
\put(1824,-1621){\makebox(0,0)[lb]{\smash{{\SetFigFont{9}{10.8}{\rmdefault}{\mddefault}{\updefault}{\color[rgb]{0,0,0}$o$}%
}}}}
\put(1860,705){\makebox(0,0)[lb]{\smash{{\SetFigFont{9}{10.8}{\rmdefault}{\mddefault}{\updefault}{\color[rgb]{0,0,0}$o_{ijk}$}%
}}}}
\put(1833,1708){\makebox(0,0)[lb]{\smash{{\SetFigFont{9}{10.8}{\rmdefault}{\mddefault}{\updefault}{\color[rgb]{0,0,0}$n_{ijk}$}%
}}}}
\put(3812,625){\makebox(0,0)[lb]{\smash{{\SetFigFont{9}{10.8}{\rmdefault}{\mddefault}{\updefault}{\color[rgb]{0,0,0}$n_{ij}$}%
}}}}
\put(2912,558){\makebox(0,0)[lb]{\smash{{\SetFigFont{9}{10.8}{\rmdefault}{\mddefault}{\updefault}{\color[rgb]{0,0,0}$p_j$}%
}}}}
\put(3182,-653){\makebox(0,0)[lb]{\smash{{\SetFigFont{9}{10.8}{\rmdefault}{\mddefault}{\updefault}{\color[rgb]{0,0,0}$p_i$}%
}}}}
\end{picture}%
\end{center}
\caption{To the proof of Lemma 2.2.}
\label{fig:DualPolyh}
\end{figure}

Equation \eqref{eqn:DualEdgeSpher} implies
\[
p_i \times \sum_j c_{ij} p_j = \sum_j c_{ij} (p_i \times p_j) = \sum (n_{ijk} - n_{ijl}) = 0,
\]
hence $\sum_j c_{ij} p_j$ is collinear with $p_i$.
To compute the scalar factor, take the inner product:
\begin{multline*}
\left\langle p_i, \sum_j c_{ij} p_j \right\rangle = \sum_j c_{ij} \langle p_i, p_j \rangle = \sum_j c_{ij} \cos\lambda_{ij}\\
= \sum_j c_{ij} \left( 1 - 2\sin^2 \frac{\lambda_{ij}}2 \right) = - 2 d_i + \sum_j c_{ij}.
\end{multline*}
Equation \eqref{eqn:DualFaceSpher} follows.

Finally, the signed area of the triangle $p_in_{ijk}n_{ijl}$ is equal to
\[
\pm \frac12 \|n_{ijk} - n_{ijl}\| \|n_{ij} - p_i\|,
\]
where the sign corresponds to the sign of $c_{ij}$.
Due to \eqref{eqn:DualEdgeSpher} the $\pm \|n_{ijk} - n_{ijl}\|$ factor equals $c_{ij} \sin\lambda_{ij}$, so that the algebraic area of the dual face is
\[
\frac12 \sum_j c_{ij} \sin\lambda_{ij} \tan \frac{\lambda_{ij}}2 = \sum_j c_{ij} \sin^2 \frac{\lambda_{ij}}2 = d_i,
\]
and the lemma is proved.
\end{proof}

The geometric interpretation of the spherical edge weight in Lemma \ref{lem:WeightsGeom} resembles that of the Euclidean edge weight as described by equation~\eqref{eqn:CotWeightGeom}.
The spherical vertex weight is similar to the area of the dual face (Voronoi cell if the triangulation is Delaunay) in the Euclidean case.

\begin{remark}
Formula for $c_{ij}$ of Lemma \ref{lem:WeightsGeom} appears in \cite{Lovasz2001} (without the restriction on the vectors $p_i$ to have unit norm) as formula for the entries of a Colin de Verdi\`ere matrix.
A generalization to higher dimensions is given in \cite{Izmestiev2010}.
\end{remark}

\begin{remark}
If the total angle around the vertex $i$ is different from $2\pi$, then it is still possible to define the algebraic area of the dual face as the sum of signed areas of triangles $p_in_{ijk}n_{ijl}$, and this area is equal to $d_i$.
Equation \eqref{eqn:DualFaceSpher} does not make sense anymore, but the obvious equation $\sum_j c_{ij} \cos\lambda_{ij} = -2d_i + \sum_j c_{ij}$ still holds.

\end{remark}

\subsection{Surface area of a convex polyhedron}
\label{sec:SurfArea}
In Theorem \ref{thm:AreaDiscrete} one deals with a convex polyhedron described by means of the outward unit normals to its faces and signed distances of the faces from the origin:
\begin{equation}
\label{eqn:Kh}
P(h) = \{x \in \mathbb{R}^3 \mid \langle x, p_i \rangle \le h_i \text{ for all }i\}.
\end{equation}
View $p_1, \ldots, p_n$ as points on $\mathbb{S}^2$ and connect two points $p_i$ and $p_j$ by the shortest great circle arc if the corresponding faces of the polyhedron share an edge.
This results in a geodesic subdivision (almost always a triangulation) of $\mathbb{S}^2$, which is called the \emph{weighted Delaunay subdivision} corresponding to the chosen weights $h_1, \ldots, h_n$.
A geodesic subdivision is called \emph{regular} if it can be realized as a weighted Delaunay subdivision for some choice of weights.
Not every geodesic subdivision is regular (even when it is a triangulation), see \cite{FillastreIzmestiev2017}.
It is shown in \cite{McMullen1973, FillastreIzmestiev2017} that for every collection of unit vectors $p_1, \ldots, p_n$ the admissible support numbers $h_1, \ldots, h_n$ (those for which none of the inequalities in \eqref{eqn:Kh} is redundant) form a convex polyhedral cone in $\mathbb{R}^n$.
This cone is further subdivided into so called \emph{type cones} corresponding to different combinatorics of the polyhedron $P(h)$.
Every $h$ from the cone of admissible support numbers belongs to some full-dimensional type cone $\mathcal{T}$.
Interior points of $\mathcal{T}$ correspond to simple polyhedra (three faces meeting at every vertex), that is to weighted Delaunay triangulations, points on the boundary of $\mathcal{T}$ correspond to removal of some of the edges of that triangulation resulting in weighted Delaunay subdivisions.

If $h_i = 1$ for all $i$, that is the polyhedron $P(h)$ is circumscribed about $\mathbb{S}^2$, then its vertices, edges and faces are the dual vertices $n_{ijk}$, dual edges and dual faces of the Delaunay subdivision of $\mathbb{S}^2$ with the vertices $p_1, \ldots, p_n$, see the paragraphs preceding Lemma \ref{lem:WeightsGeom}.
The dual vertices project to the circumcenters of triangles, the dual edges to the edges of the Voronoi cells.
The formula of Theorem \ref{thm:AreaDiscrete} says in this case that the area of the circumscribed polyhedron is equal to $\frac12 \sum_i d_i$, which follows directly from Lemma \ref{lem:WeightsGeom}, as $d_i$ is the area of the $i$-th dual face.

At the same time, for a general $h$ the triangulation used in Theorem \ref{thm:AreaDiscrete} is not necessarily Delaunay, so that some of the edge weights may be negative.
This shows that the discrete spherical Laplacian makes sense not only for Delaunay triangulations, but at least for regular triangulations as well.

\begin{proof}[Proof of Theorem \ref{thm:AreaDiscrete}]
Fix a type cone $\mathcal{T}$ to which our support numbers $h$ belong and let $h$ vary over $\mathcal{T}$.
Denote by $A_i(h)$ the area of the $i$-th face, by $A(h) = \sum_i A_i(h)$ the total surface area, and by $V(h)$ the volume of $P(h)$.
Simple geometric arguments illustrated by Figure \ref{fig:VolPartDer} show that
\[
\frac{\partial V}{\partial h_i} = A_i
\]
\begin{equation}
\label{eqn:DA}
\frac{\partial^2 V}{\partial h_i \partial h_j} = \frac{\partial A_i}{\partial h_j} = \frac{L_{ij}}{\sin\lambda_{ij}} \text{ for }i \ne j, \quad \frac{\partial^2 V}{\partial h_i^2} = \frac{\partial A_i}{\partial h_i} = - \sum_j L_{ij} \cot\lambda_{ij},
\end{equation}
where $L_{ij}(h)$ is the length of the edge of $P(h)$ shared by the $i$-th and the $j$-th faces.

\begin{figure}
\begin{center}
\includegraphics[width=.8\textwidth]{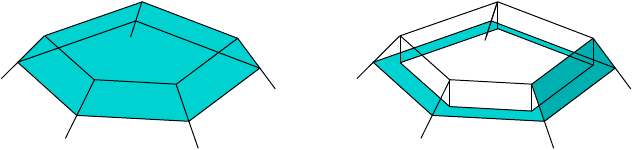}
\end{center}
\caption{Partial derivatives of the volume and the face areas. The increments of the volume and of the face areas are shaded.}
\label{fig:VolPartDer}
\end{figure}

At the same time the functions $A(h)$ and $V(h)$ are homogeneous polynomials in $h_1, \ldots, h_n$ of degrees $2$ and $3$, respectively.
This follows from the formulas
\[
V(h) = \frac13 \sum_i h_i A_i(h), \quad A_i(h) = \frac12 \sum_j h_{ij}(h) L_{ij}(h),
\]
where $h_{ij}$ is the distance from the projection of the origin on the plane of the $i$-th face to the line containing the $ij$-th edge, and from the fact that $h_{ij}$ and $L_{ij}$ are linear functions of $h$.
Let
\[
V(h) = \sum_{i,j,k} b_{ijk} h_ih_jh_k, \quad A(h) = \sum_{i,j} a_{ij} h_ih_j
\]
with $b_{\sigma(i)\sigma(j)\sigma(k)} = b_{ijk}$ for all permutations $\sigma \in S_3$ and $a_{ij} = a_{ji}$.
Differentiating the polynomial $V(h)$ one obtains
\begin{gather*}
\frac{\partial V}{\partial h_i} = \sum_{j,k} b_{ijk} h_jh_k + \sum_{j,k} b_{jik} h_jh_k + \sum_{j,k} b_{jki} h_jh_k = 3 \sum_{j,k} b_{ijk} h_jh_k,\\
\frac{\partial^2 V}{\partial h_i \partial h_j} = 6 \sum_k b_{ijk} h_k,
\end{gather*}
which implies
\[
A(h) = \sum_{i=1}^n A_i(h) = \sum_{i=1}^n \frac{\partial V}{\partial h_i} = 3\sum_{i,j,k} b_{ijk}h_jh_k = 3 \sum_{i,j,k} b_{ijk}h_ih_j
\]
and hence
\[
a_{ij} = 3\sum_k b_{ijk} = \frac12 \left. \frac{\partial^2V}{\partial h_i \partial h_j} \right|_{h=1}.
\]
(Observe that although the point $h_i = 1$ for all $i$ does not necessarily belong to our type cone so that formally the function $V(h)$ may be undefined at this point, one can substitute $h=1$ into the polynomial expression for $V(h)$.)
Together with equation \eqref{eqn:DA} this implies

\begin{align}\label{eqn:aij}
	\begin{split}
		a_{ij} &= \frac{L_{ij}(1)}{2\sin\lambda_{ij}} = \frac{c_{ij}}2 \text{ for }i \ne j,\\
		a_{ii} &= -\frac12 \sum_j L_{ij}(1) \cot\lambda_{ij} = -\frac12 \sum_j c_{ij} \cos\lambda_{ij},
	\end{split}
\end{align}
as for $h=1$ the edge of $P(1)$ shared by the $i$-th and the $j$-th faces is exactly the dual edge $n_{ijk}n_{ijl}$ from the paragraphs preceding Lemma \ref{lem:WeightsGeom}.

On the other hand one has
\begin{multline*}
\frac12 \langle h, 2h + \triangle_sh \rangle_d = \sum_{i=1}^n d_i h_i^2 + \frac12 \sum_{i=1}^n h_i \sum_{j=1}^n c_{ij}(h_j - h_i)\\
= \sum_{i=1}^n \left( d_i - \frac12 \sum_{j=1}^n c_{ij} \right) h_i^2 + \sum_{\{i,j\}} c_{ij}h_ih_j,
\end{multline*}
and
\[
d_i - \frac12 \sum_{j=1}^n c_{ij} = \sum_{j \ne i} c_{ij} \left(\sin^2 \frac{\lambda_{ij}}2 - \frac12 \right) = -\frac12 \sum_{j \ne i} c_{ij} \cos\lambda_{ij}.
\]
Thus in view of \eqref{eqn:aij} one has
\[
A(h) = \sum_{i,j} a_{ij} h_ih_j = \sum_i a_{ii} h_i^2 + 2 \sum_{\{i,j\}} a_{ij} h_ih_j = \frac12 \langle h, 2h + \triangle_sh \rangle_d,
\]
and the theorem is proved.
\end{proof}

\begin{remark}
The vertex weight $d_i$ of a non-Delaunay triangulation may be negative or, even worse, may vanish.
In the latter case the formula \eqref{eqn:DiscLapl} does not make sense.
The weak Laplacian however remains defined, and the formula for $\langle h, 2h + \triangle_s h \rangle_d$ contains $d_i$ but does not involve division by zero.
\end{remark}

\subsection{Infinitesimal isometric deformations of inscribed polyhedra}
\label{sec:InfDef}
Let $(V, E)$ be a graph.
A \emph{realization} of $(V, E)$ in $\mathbb{R}^3$ is a map $p \colon V \to \mathbb{R}^3$ such that $p_i \ne p_j$ for all $ij \in E$.
An \emph{infinitesimal deformation} of a realization $p$ is an assignment $q \colon V \to \mathbb{R}^3$ of a vector $q_i$ to every vertex $i$.
An infinitesimal deformation is called \emph{isometric} if for every edge $ij$ one has
\[
\langle p_j - p_i, q_j - q_i \rangle = 0.
\]
This equation is equivalent to
\[
\left. \frac{d}{dt} \right|_{t=0} \|p_j(t) - p_i(t)\| = 0
\]
for any one-parameter family of realizations $p(t)$ such that $p(0) = p$ and $p'(0) = q$.
The restriction of any Killing field on $\mathbb{R}^3$
\[
q_i = r \times p_i + t, \quad r, t \in \mathbb{R}^3,
\]
is an isometric infinitesimal deformation and is called \emph{trivial}.

Let $(V, E)$ be the graph of a triangulated surface $(V, E, F)$, and let $p$ be such realization that for all $ijk \in F$ the triangle $p_ip_jp_k$ does not degenerate to a segment.
A triangle in $\mathbb{R}^3$ is infinitesimally rigid, that is all of its isometric infinitesimal deformations are trivial.
Therefore to every infinitesimal isometric deformation of $(V, E)$ one can associate a \emph{rotation field} and a \emph{translation field}, which are maps
\[
r \colon F \to \mathbb{R}^3, \quad t \colon F \to \mathbb{R}^3
\]
such that for every vertex $i$ of every face $ijk$ one has
\begin{equation}
\label{eqn:RotTranslFields}
q_i = r_{ijk} \times p_i + t_{ijk}.
\end{equation}

\begin{lemma}
\label{lem:RotTranslProp}
Let $q$ be an infinitesimal isometric deformation of an $\mathbb{R}^3$-realization $p$ of a triangulated surface $(V, E, F)$, and let $r$ and $t$ be its rotation and translation fields.
Then there is a function $s \colon E \to \mathbb{R}$ such that for every edge $ij$ and its two adjacent triangles $ijk$ and $ijl$ one has
\begin{equation}
\label{eqn:RTStress}
r_{ijk} - r_{ijl} = s_{ij}(p_j - p_i), \quad t_{ijk} - t_{ijl} = s_{ij} (p_i \times p_j)
\end{equation}
and for all vertices $i$ one has
\begin{equation}
\label{eqn:Stress}
\sum_j s_{ij} (p_j - p_i) = 0.
\end{equation}

Conversely, if $p \colon V \to \mathbb{R}^3$ is a realization of a simply-connected triangulated surface $(V, E, F)$ and $s \colon E \to \mathbb{R}$ is a function satisfying \eqref{eqn:Stress}, then there is an infinitesimal isometric deformation $q$ of $p$, unique modulo trivial infinitesimal isometric deformations, such that its rotation and translation fields satisfy \eqref{eqn:RTStress}.
\end{lemma}

A function $s$ with the property \eqref{eqn:Stress} is called a \emph{self-stress} on the graph of the realization $v$.
The above lemma is classical and goes back probably to the XIX century.
We provide a proof for completeness.

\begin{proof}
The rotation and translation vectors of the triangles $ijk$ and $ijl$ allow to represent each of the vectors $q_i, q_j$ in two ways:
\begin{gather*}
r_{ijk} \times p_i + t_{ijk} = q_i = r_{ijl} \times p_i + t_{ijl},\\
r_{ijk} \times p_j + t_{ijk} = q_j = r_{ijl} \times p_j + t_{ijl}.
\end{gather*}
Taking differences results in
\[
(r_{ijk} - r_{ijl}) \times p_i = t_{ijl} - t_{ijk} = (r_{ijk} - r_{ijl}) \times p_j,
\]
which implies that $r_{ijk} - r_{ijl}$ is parallel to $p_j - p_i$, thus there is $s \colon E \to \mathbb{R}$ such that $r_{ijk} - r_{ijl} = s_{ij}(p_j - p_i)$.
Substituting this into the last displayed equation one obtains $t_{ijk} - t_{ijl} = s_{ij}(p_i \times p_j)$, and the first part of the lemma is proved.

For the second part first choose the value of $r$ at one of the triangles $\Delta_0 \in F$ arbitrarily.
Then, for any other triangle $\Delta \in F$ choose a path across edges from $\Delta_0$ to $\Delta$ and extend $r$ along this path according to the first formula in \eqref{eqn:RTStress}.
The result is independent on the choice of the path due to the closing condition \eqref{eqn:Stress} and to the simple connectedness of the surface.
The same works for $t$: its value at one of the triangles can be chosen arbitrarily, and the closing condition \eqref{eqn:Stress} ensures that $t$ can be extended to all of $F$ so that the second equation in \eqref{eqn:RTStress} holds.
Having defined $r, t \colon F \to \mathbb{R}^3$ this way, define $q \colon V \to \mathbb{R}^3$ as in \eqref{eqn:RotTranslFields}.
This $q$ is well-defined (that is, the value of $q_i$ is independent on the choice of an edge $jk$ forming a triangle with $i$) due to \eqref{eqn:RTStress}.
Besides, this $q$ is an infinitesimal isometric deformation of $p$ as \eqref{eqn:RTStress} implies $q_j - q_i = s_{ij} (p_j - p_i) \times p_i$, which is orthogonal to $p_j - p_i$.
\end{proof}

\begin{proof}[Proof of Theorem \ref{thm:SphIIDDiscrete}, part 1]
The radial component $f$ of $q$ is equal to
\[
f_i = \langle q_i, p_i \rangle.
\]
Decompose the difference of the values of $f$ at a pair of adjacent vertices as
\[
f_j - f_i = \langle q_j, p_j \rangle - \langle q_i, p_i \rangle = \langle p_j, q_j - q_i \rangle + \langle q_i, p_j - p_i \rangle.
\]
This allows to write the Laplacian as
\begin{equation}
\label{eqn:LaplParts}
(\triangle_s f)_i = \frac{1}{d_i} \sum_j c_{ij} \langle p_j, q_j - q_i \rangle + \left\langle \frac{q_i}{d_i}, \sum_j c_{ij}(p_j - p_i) \right\rangle
\end{equation}
Due to $q_j - q_i = r_{ijk} \times (p_j - p_i)$ one has
\[
\langle p_j, q_j - q_i \rangle = \det(p_j, r_{ijk}, p_j - p_i) = -\det(p_i, p_j, r_{ijk}) = - \langle p_i \times p_j, r_{ijk} \rangle,
\]
so that because of \eqref{eqn:DualEdgeSpher} the first sum in \eqref{eqn:LaplParts} transforms as
\begin{multline*}
\sum_j c_{ij} \langle p_j, q_j - q_i \rangle = - \sum_j c_{ij} \langle p_i \times p_j, r_{ijk} \rangle\\
= - \sum_j \langle n_{ijk} - n_{ijl}, r_{ijk} \rangle = -\sum_j \langle n_{ijk}, r_{ijk} - r_{ijl} \rangle = 0
\end{multline*}
because $n_{ijk}$ is orthogonal to $p_j - p_i$ which by \eqref{eqn:RTStress} is parallel to $r_{ijk} - r_{ijl}$.

To compute the second sum in \eqref{eqn:LaplParts} use \eqref{eqn:DualFaceSpher}.
It implies
\[
\sum_j c_{ij} (p_j - p_i) = -2d_ip_i,
\]
and therefore
\[
(\triangle_s f)_i = \left\langle \frac{q_i}{d_i}, -2d_ip_i \right\rangle = -2 \langle q_i, p_i \rangle = -2f_i,
\]
which concludes the proof.
\end{proof}

For the second part of Theorem \ref{thm:SphIIDDiscrete} some preliminaries are needed.
As in the paragraphs preceding Lemma \ref{lem:WeightsGeom} construct the dual polyhedral surface with faces in the planes
\[
p_i^\ast = \{x \mid \langle p_i, x \rangle = 1\}
\]
and vertices $n_{ijk}$, which are the intersection points of $p_i^\ast, p_j^\ast, p_k^\ast$.
The dual faces are in general non-convex and self-intersecting but have well-defined algebraic areas $A_i$.
For a function $f \colon V \to \mathbb{R}$ consider a 1-parameter deformation of the dual polyhedron obtained by moving the dual planes with the velocities $f_i$ away or, respectively, towards the origin:
\begin{equation}
\label{eqn:DeformDualPolyh}
p_i^\ast(\epsilon) = \{x \mid \langle p_i, x \rangle = 1 + \epsilon f_i\}.
\end{equation}
The combinatorics of the dual polyhedron is kept fixed: to every face $ijk \in F$ one associates a dual vertex $n_{ijk}(\epsilon)$, which is the intersection point of $p_i^\ast(\epsilon), p_j^\ast(\epsilon), p_k^\ast(\epsilon)$, and to every edge $ij \in E$ one associates a dual edge $n_{ijk}(\epsilon)n_{ijl}(\epsilon)$.
Let $A_i(\epsilon)$ be the algebraic area of the $i$-th face of the deformed dual polyhedron.
By definition one has
\begin{equation}
\label{eqn:AlgArea}
A_i(\epsilon) = \frac12 \sum_j h_{ij}(\epsilon) L_{ij}(\epsilon),
\end{equation}
where $h_{ij}(\epsilon)$ is the signed distance from the projection of $o$ to $p_i^\ast(\epsilon)$ to the line $p_i^\ast(\epsilon) \cap p_j^\ast(\epsilon)$, and $L_{ij}(\epsilon)$ is the signed length of the dual edge $n_{ijk}(\epsilon)n_{ijl}(\epsilon)$.
One has $h_{ij}(0) = \|n_{ij} - p_j\| > 0$, therefore $h_{ij}(\epsilon)$ remains positive for small $\epsilon$, the sign of $L_{ij}(\epsilon)$ is positive if and only if $n_{ijk} - n_{ijl}$ is a positive multiple of $p_i \times p_j$ (assuming the triangle $p_ip_jp_k$ is positively oriented).

\begin{lemma}
\label{lem:ADer}
The derivative of $A_i$ as a function of $\epsilon$ satisfies
\[
A'_i = \sum_j h_{ij} L'_{ij} = \sum_j h'_{ij} L_{ij}.
\]
Besides, $(\triangle_s f)_i = -2f_i$ implies $A'_i(0) = 0$.
\end{lemma}
\begin{proof}
The area $A_i$ can be viewed as a function of $h_{ij}$.
As such, it satisfies
\[
\frac{\partial A_i}{\partial h_{ij}} = L_{ij},
\]
which is a 2-dimensional analog of the formulas illustrated in Figure \ref{fig:VolPartDer}.
It follows that
\[
A'_i = \sum_j \frac{\partial A_i}{\partial h_{ij}} h'_{ij} = \sum_j h'_{ij} L_{ij}.
\]
On the other hand, formula \eqref{eqn:AlgArea} implies
\[
A'_i = \frac12 \sum_j h'_{ij} L_{ij} + \frac12 \sum_j h_{ij} L'_{ij}.
\]
Combined with the previous equation it yields $A'_i = \sum_j h_{ij} L'_{ij}$.

Now view $A_i$ as a function of the support numbers $h_i$, which themselves are functions of $\epsilon$:
\[
h_i(\epsilon) = 1 + \epsilon f_i.
\]
One has
\[
A'_i = \frac{\partial A_i}{\partial h_i} f_i + \sum_j \frac{\partial A_i}{\partial h_j} f_j = - f_i \sum_j L_{ij} \cot\lambda_{ij} + \sum_j f_j \frac{L_{ij}}{\sin\lambda_{ij}}.
\]
As $L_{ij}(0) = c_{ij} \sin\lambda_{ij}$, this implies
\[
A'(0) = - f_i \sum_j c_{ij}\cos\lambda_{ij} + \sum_j f_j c_{ij} = d_i (2f_i + (\triangle_s f)_i),
\]
and the lemma follows.
\end{proof}

\begin{proof}[Proof of the second part of Theorem \ref{thm:SphIIDDiscrete}]
Let $f \colon V \to \mathbb{R}$ be such that $\triangle_s f = -2f$.
Consider the $1$-parameter deformation \eqref{eqn:DeformDualPolyh} of the dual polyhedron and put
\[
t_{ijk} = n'_{ijk}.
\]
Due to
\[
\langle p_i, n_{ijk}(\epsilon) \rangle = 1 + \epsilon f_i, \quad \langle p_j, n_{ijk}(\epsilon) \rangle = 1 + \epsilon f_j, \quad \langle p_k, n_{ijk}(\epsilon) \rangle = 1 + \epsilon f_k 
\]
one has
\[
\langle p_i, t_{ijk} \rangle = f_i, \quad \langle p_j, t_{ijk} \rangle = f_j, \quad \langle p_k, t_{ijk} \rangle = f_k. 
\]
It follows that
\[
\langle p_i, t_{ijk} - t_{ijl} \rangle = f_i - f_i = 0, \quad \langle p_j, t_{ijk} - t_{ijl} \rangle = f_j - f_j = 0.
\]
Therefore the vector $t_{ijk} - t_{ijl}$ is parallel to the vector $p_i \times p_j$ and hence there is a map $s \colon E \to \mathbb{R}$ such that
\[
t_{ijk} - t_{ijl} = s_{ij}(p_i \times p_j).
\]
Let us prove that $s$ satisfies \eqref{eqn:Stress}.
The assumption $\triangle_s f = -2f$ implies by Lemma \ref{lem:ADer}
\begin{equation}
\label{eqn:APrime}
0 = A_i'(0) = \sum_j h_{ij}(0) L'_{ij}(0).
\end{equation}
One has $h_{ij}(0) = \|n_{ij} - p_i\| = \tan\frac{\lambda_{ij}}2$.
As $L_{ij}(\epsilon)$ is the signed length of the segment $n_{ijk}(\epsilon)n_{ijl}(\epsilon)$ and
\[
n_{ijk}(\epsilon) - n_{ijl}(\epsilon) = (n_{ijk} + \epsilon t_{ijk}) - (n_{ijl} + \epsilon t_{ijl}) = (c_{ij} + \epsilon s_{ij}) (p_i \times p_j),
\]
one also has $L'_{ij} = s_i \sin\lambda_{ij}$.
Substituting this into \eqref{eqn:APrime} one obtains
\[
\sum_j s_{ij} \sin^2 \frac{\lambda_{ij}}2 = 0.
\]
It follows that
\[
\left\langle p_i, \sum_j s_{ij}(p_j - p_i) \right\rangle = \sum_j s_{ij}(\cos\lambda_{ij} - 1) = 0.
\]
Since one also has
\[
p_i \times \sum_j s_{ij}(p_j - p_i) = \sum_j (t_{ijk} - t_{ijl}) = 0,
\]
one concludes that $\sum_j s_{ij}(p_j - p_i) = 0$.

The closing condition ensures that there is a map $r \colon F \to \mathbb{R}^3$ satisfying the first of the equations \eqref{eqn:RTStress}.
Putting $q_i = r_{ijk} \times p_i + t_{ijk}$ one obtains an infinitesimal isometric deformation of the polyhedral surface $P$ such that $\langle p_i, q_i \rangle = \langle p_i, t_{ijk} \rangle = f_i$, and the second part of Theorem \ref{thm:SphIIDDiscrete} is proved.
\end{proof}

\begin{remark}
Some parts of the above arguments are related to the Maxwell--Cremona correspondence, reciprocal diagrams and parallel redrawings, see \cite{CrapoWhiteley1994}.
The connection between infinitesimal isometric deformations of polyhedra and deformations of their duals preserving the face areas was discovered in \cite{BobenkoIzmestiev2008}.
\end{remark}

\subsection{Infinitesimal conformal deformations}
\label{sec:InfConfDef}
Consider two triangulated spherical surfaces with the same combinatorics $(V, E, F)$ and the edge lengths $\lambda_{ij}$, respectively $\widetilde{\lambda}_{ij}$.
They are called \emph{conformally equivalent} (in the vertex scaling sense) if there is a function $u \colon V \to \mathbb{R}$ such that
\begin{equation}
\label{eqn:DiscConfSph}
\sin \frac{\tilde{\lambda}_{ij}}{2} = e^{\frac{u_i+u_j}{2}}\sin \frac{\lambda_{ij}}{2} 
\end{equation}
for all edges $ij \in E$.
This notion together with its hyperbolic counterpart was introduced in \cite{Bobenko2010} motivated by variational principles for 3-dimensional polyhedra.
A similar formula for triangulated Euclidean surfaces was proposed earlier in \cite{Luo2004}. See \cite{Glickenstein2017, Zhang2014} for further variants.

Let us define infinitesimal conformal deformations in a way compatible with the above formula.

\begin{definition}
An \emph{infinitesimal conformal deformation} (in the vertex scaling sense)  of a triangulated spherical surface with edge lengths $\lambda_{ij}$ is a map $\dot\lambda \colon E \to \mathbb{R}$ such that $\dot\lambda_{ij}$ are the derivatives at $t=0$ of the edge lengths under some discrete conformal deformation.
In other words, the numbers $\dot{\lambda}_{ij}$ must satisfy
\begin{equation}
\label{eqn:InfConf}
\left( \log\sin\frac{\lambda_{ij}}2 \right)^{\boldsymbol{\cdot}} = \frac{u_i + u_j}2 \quad \text{or, equivalently} \quad \dot\lambda_{ij} = (u_i + u_j) \tan\frac{\lambda_{ij}}2
\end{equation}
for some function $u \colon V \to \mathbb{R}$.
The numbers $u_i$ are called \emph{conformal factors} of the infinitesimal conformal deformation $\dot\lambda$.
\end{definition}

We are dealing with deformations where the total angles around all vertices remain $2\pi$.
Therefore let us map our triangulated surface to $\mathbb{S}^2$ in a locally isometric way.
This map is uniquely determined by its restriction $p \colon V \to \mathbb{S}^2$ to the vertices.
Any deformation of our surface in the class of spherical surfaces can be described by a collection of paths $p_i(t)$, $i \in V$.
An infinitesimal deformation is the derivative at $t=0$ of a smooth deformation, that is a collection of tangent vectors $\xi_i \in T_{p_i}\mathbb{S}^2$.
A tangent vector field $\xi$ defined on the vertices of a triangulated spherical surface is called conformal if the induced infinitesimal deformation of the edge lengths is conformal.

\begin{definition}
\label{dfn:DiscConfField}
Let $p \colon V \to \mathbb{S}^2$ be a locally isometric realization of a triangulated spherical surface.
A tangent vector field
\[
\xi \colon V \to T\mathbb{S}^2, \quad \xi_i \in T_{p_i}\mathbb{S}^2
\]
is called \emph{discretely conformal} (in the sense of vertex scaling) with conformal factors $u \colon V \to \mathbb{R}$, if for every smooth deformation $p(t)$ with $p(0) = p$ and $\dot p(0) = \xi$ the derivatives at $t=0$ of the edge lengths satisfy \eqref{eqn:InfConf} for some function $u \colon V \to \mathbb{R}$.
\end{definition}

\begin{lemma}
\label{lem:DiscConf}
A tangent vector field $\xi$ on the vertices of a triangulated spherical surface, locally isometrically mapped to $\mathbb{S}^2$ is discretely conformal with conformal factors $u \colon V \to \mathbb{R}$ if and only if there is a function $u \colon V \to \mathbb{R}$ such that for every edge $ij$ one has
\[
\langle \xi_j - \xi_i, p_j - p_i \rangle = \frac{u_i + u_j}2 \|p_j - p_i\|^2.
\]
\end{lemma}
\begin{proof}
By differentiating the identity $\cos \lambda_{ij} = \langle p_i, p_j \rangle$ one obtains
\[
- \dot\lambda_{ij} \sin\lambda_{ij} = \langle \xi_i, p_j \rangle + \langle p_i, \xi_j \rangle = - \langle \xi_j - \xi_i, p_j - p_i \rangle.
\]
The formula in the lemma follows by substituting the expression for $\dot\lambda_{ij}$ from \eqref{eqn:InfConf} and using the identity $\|p_j - p_i\| = 2\sin\frac{\lambda_{ij}}2$.
\end{proof}

\begin{proof}[Proof of Theorem \ref{thm:ConfFieldDiscreteSph}, first part and of Theorem \ref{thm:SphIIDDiscrete}, third part]
Let $\xi \colon V \to T\mathbb{S}^2$ be a tangent vector field such that the induced derivatives of the edge lengths satisfy \eqref{eqn:InfConf}.
Consider the vector field $q \colon V \to \mathbb{R}^3$ defined as
\[
q_i = -u_i p_i + \xi_i.
\]
Let us show that $q$ is an infinitesimal isometric deformation of the Euclidean triangulated surface with the same vertices and the same edges.
This is done by a short calculation:
\begin{multline*}
\langle q_j - q_i, p_j - p_i \rangle = \langle u_ip_i - u_jp_j, p_j - p_i \rangle + \langle \xi_j - \xi_i, p_j - p_i \rangle\\
= (u_i + u_j)\left( \langle p_i, p_j \rangle - 1 \right) + \frac{u_i + u_j}2 \|p_j - p_i\|^2 = 0,
\end{multline*}
based on $\|p_i\| = \|p_j\| = 1$ and on Lemma \ref{lem:DiscConf}.
By the first part of Theorem \ref{thm:SphIIDDiscrete} proved in the previous section the radial component of $q$ is an eigenfunction of the discrete spherical Laplacian with the eigenvalue $-2$.
This proves that condition 2 in Theorem \ref{thm:ConfFieldDiscreteSph} implies condition 1.

For the proof that condition 1 in Theorem \ref{thm:ConfFieldDiscreteSph} implies condition 2 we use the second part of Theorem \ref{thm:SphIIDDiscrete}.
For a function $u \colon V \to \mathbb{R}$ such that $\triangle_s u = -2u$ there is an infinitesimal isometric deformation $\xi$ whose radial component is equal to $u$.
The above calculation shows that minus the tangential component of $\xi$ satisfies the formula from Lemma \ref{lem:DiscConf} and hence the formula \eqref{eqn:InfConf}.

Now to the third part of Theorem \ref{thm:SphIIDDiscrete}.
Let $\xi$ be the tangential component of an infinitesimal isometric deformation $q$ of a triangulated Euclidean surface inscribed in the unit sphere.
The same calculation as above implies
\[
\langle \xi_j - \xi_i, p_j - p_i \rangle = - \frac{\langle q_i, p_i \rangle + \langle q_j, p_j \rangle}2 \|p_j - p_i\|^2,
\]
which by Lemma \ref{lem:DiscConf} means that $\xi$ is discrete conformal with the conformal factor minus the radial component of $q$.
This finishes the proof.
\end{proof}

Definition of discrete conformal equivalence at the beginning of this section can be reformulated in terms of \emph{length cross ratios} \cite{Bobenko2010}. For any four points $p_1, p_2, p_3, p_4 \in \mathbb{R}^3$ the length cross-ratio is defined as
\[
\operatorname{lcr}(p_1, p_2, p_3, p_4) = \frac{\|p_1 - p_3\| \|p_2 - p_4\|}{\|p_1 - p_4\| \|p_2 - p_3\|}.
\]

\begin{lemma}[\cite{Bobenko2010}]
There is a function $u \colon V \to \mathbb{R}$ such that equation \eqref{eqn:DiscConfSph} holds for all edges $ij \in E$ if and only if for every edge $ij$ one has
\[
\operatorname{lcr}(p_i, p_j, p_k, p_l) = \operatorname{lcr}(\widetilde{p}_i, \widetilde{p}_j, \widetilde{p}_k, \widetilde{p}_l),
\]
where $ijk, ijl \in F$.
\end{lemma}
The proof in \cite{Bobenko2010} is based on the observation that for $p_1, p_2, p_3, p_4 \in \mathbb{S}^2$ one has
\[
\operatorname{lcr}(p_1, p_2, p_3, p_4) = \frac{\sin \frac{\lambda_{13}}{2} \sin \frac{\lambda_{24}}{2}}{\sin \frac{\lambda_{14}}{2} \sin \frac{\lambda_{23}}{2}}
\]
and can be found e.g. in \cite[Theorem 2.4]{Lam2015a}. The length cross ratios are invariant under M\"{o}bius transformations.

For four points $p_1, p_2, p_3, p_4 \in \mathbb{S}^2$ lying on a sphere,  one can further define a complex-valued cross-ratio as
\[
\operatorname{cr}(p_1, p_2, p_3, p_4) = \operatorname{cr}(z_1, z_2, z_3, z_4) = \frac{(z_1 - z_3)(z_2 - z_4)}{(z_1 - z_4)(z_2 - z_3)},
\]
where $z_i \in \mathbb{C}$ is the image of $p_i$ under a stereographical projection of $\mathbb{S}^2$ to $\mathbb{R}^2 \cup \{\infty\}$ identified with $\mathbb{C} \cup \{\infty\}$.

\begin{lemma}
The complex cross-ratio of four points on the sphere is well-defined in the sense that it does not depend on the choice of the pole of the stereographic projection.
\end{lemma}
\begin{proof}
The absolute value of the complex cross-ratio is the length cross-ratio.
The stereographic projection is a restriction of the inversion in a sphere, and it can be shown by considering pairs of similar triangles that the length cross-ratio is invariant under inversion.

The argument of the complex cross-ratio is the angle between oriented circumcircles of triangles $z_1z_2z_3$ and $z_1z_2z_4$ .
At the same time, the stereographic projection maps circles to circles and preserves oriented angles.
Therefore the argument of $\operatorname{cr}(p_1, p_2, p_3, p_4)$ is also the angle between oriented circumcircles.
Thus both the absolute value and the argument of $\operatorname{cr}(p_1, p_2, p_3, p_4)$ are well-defined.

Alternatively one may use the observation that two stereographic projections from different poles differ by a M\"{o}bius transformation, which preserves the cross-ratio.
\end{proof}

Given a triangulation of a spherical domain and a stereographic projection, we could use the complex cross-ratios to determine the positions of the vertices up to a M\"obius transformation of the sphere. Following \cite{lam2024sym}, we define $X:E \to \mathbb{C}$ such that for every oriented edge $ij$
\[
X_{ij}:= -\operatorname{cr}(z_i,z_j,z_k,z_l)
\]
where $z_i z_j z_k$ is the left triangle and $z_i z_j z_l$ is on the right. It turns out $X_{ij}=X_{ji}$ and hence $X$ is well-defined on unoriented edges. Observe that the absolute value $|X_{ij}| =\operatorname{lcr}(p_i, p_j, p_k, p_l) $ is the length cross ratio while the argument $\arg X_{ij}$ is the intersection angle of neighboring circumcircles.

The second part of Theorem~\ref{thm:ConfFieldDiscreteSph} follows from the following proposition on the equivalence of discrete conformal structures under infinitesimal deformation.

\begin{proposition}
	\label{prop:conformal-equivalence}
	Let $\xi : V \to T\mathbb{S}^2$ be an infinitesimal conformal deformation in the sense of vertex scaling. Then the rotated field $\eta := J\xi$, where $J$ is a rotation by $\pi/2$ in each tangent plane, preserves the intersection angles of adjacent circumcircles infinitesimally. Conversely, any infinitesimal deformation $\eta$ preserving intersection angles arises in this way.
\end{proposition}

\begin{proof}
	Let $\phi \colon \mathbb{S}^2 \to \mathbb{C} \cup \{\infty\}$ be a stereographic projection.
	Denote $z_i = \phi(p_i)$ and $\upxi_i = d\phi(\xi_i)$.
	Then $\xi$ is a discretely conformal vector field in the sense of Definition~\ref{dfn:DiscConfField} if and only if the length cross-ratios $|X|$ are preserved infinitesimally, which is equivalent to the vanishing of the real part of the derivative of $\log X_{ij}$ for each $ij \in E$:
	\[
	0 = \left( \Re \log X_{ij} \right)^{\boldsymbol{\cdot}} = \Re\left( \frac{\upxi_k - \upxi_i}{z_k - z_i} - \frac{\upxi_i - \upxi_l}{z_i - z_l} + \frac{\upxi_l - \upxi_j}{z_l - z_j} - \frac{\upxi_j - \upxi_k}{z_j - z_k} \right).
	\]
	On the other hand, for $\upeta_i = d\phi(\eta_i)$, the vector field $\eta$ preserves the intersection angles $\arg X$ of adjacent circumcircles in the infinitesimal sense if and only if
	\[
	0 = \left( \Im \log X_{ij} \right)^{\boldsymbol{\cdot}} = \Im\left( \frac{\upeta_k - \upeta_i}{z_k - z_i} - \frac{\upeta_i - \upeta_l}{z_i - z_l} + \frac{\upeta_l - \upeta_j}{z_l - z_j} - \frac{\upeta_j - \upeta_k}{z_j - z_k} \right).
	\]
	Therefore, infinitesimal conformal deformations in the sense of vertex scaling are related to those preserving intersection angles via a rotation by $\pi/2$ in the tangent plane at each vertex.
\end{proof}

This proof basically repeats an argument from \cite{Lam2015a} in the Euclidean case, where no stereographic projection was needed.

\section{Discrete hyperbolic Laplacian}	
\label{sec:HypLap}
\subsection{Geometric meaning of the edge and vertex weights}
Let $\mathbb{R}^{2,1}$ be the Minkowski space, that is $\mathbb{R}^3$ equipped with the inner product
\[
\langle x, y \rangle_{2,1} = -x_0y_0 + x_1y_1 + x_2y_2 \quad \text{for} \quad x = (x_0, x_1, x_2),\, y = (y_0, y_1, y_2).
\]
A non-zero vector $x$ is said to be spacelike if $\langle x, x \rangle_{2,1} > 0$, timelike if $\langle x, x \rangle_{2,1} < 0$, and lightlike if $\langle x, x \rangle_{2,1} = 0$.
A one-dimensional linear subspace of $\mathbb{R}^{2,1}$ is said to be spacelike, timelike or lightlike if it spanned by a vector of the respective type.
A two-dimensional linear subspace of $\mathbb{R}^{2,1}$ is said to be spacelike, timelike or lightlike if the restriction of the Minkowski inner product has signature $(+,+)$, $(+,-)$, $(+,0)$, respectively.

Define the Minkowski vector product as
\[
\begin{pmatrix} x_0\\ x_1\\ x_2 \end{pmatrix} \times
\begin{pmatrix} y_0\\ y_1\\ y_2 \end{pmatrix} =
\begin{pmatrix}
-x_1y_2 + x_2y_1\\
x_2y_0 - x_0y_2\\
x_0y_1 - x_1y_0
\end{pmatrix}.
\]
It is easy to see that
\[
\langle x \times y, z \rangle_{2,1} = \det(x, y, z).
\]
In particular, one has
\[
\langle x \times y, x \rangle_{2,1} = 0 \quad \text{and} \quad \det(x, y, x \times y) = \langle x \times y, x \times y \rangle_{2,1},
\]
so that if $x$ and $y$ are timelike, then $x \times y$ is orthogonal to both of them (hence spacelike) and $(x, y, x \times y)$ is a positively oriented basis of $\mathbb{R}^{2,1}$.

The upper half of the one-sheeted hyperboloid in the Minkowski space serves as a model for the hyperbolic plane:
\[
\mathbb{H}^2 = \{x \in \mathbb{R}^{2,1} \mid \langle x, x \rangle_{2,1} = -1,\, x_0 > 0\}.
\]
The distance $\lambda_{12}$ between two points $p_1, p_2 \in \mathbb{H}^2$ is defined as
\[
\langle p_1, p_2 \rangle_{2,1} = -\cosh \lambda_{12}.
\]
The intersection of $\mathbb{H}^2$ with a timelike two-dimensional linear subspace is a geodesic, that is the image of an isometric embedding of $\mathbb{R}$.
In this respect $\mathbb{H}^2$ strongly resembles $\mathbb{S}^2$, where the inner product is related to the distance in a similar way, and great circles are (local) geodesics.
The angle between two curves on $\mathbb{H}^2$ is defined as the angle between their tangent vectors, which makes sense due to the fact that the tangent plane at every point of $\mathbb{H}^2$ is spacelike, thus isomorphic to $\mathbb{R}^2$ with the Euclidean inner product.

The circle with center $p \in \mathbb{H}^2$ and radius $r$ is described by the equation
\[
\langle p, x \rangle_{2,1} = -\cosh r.
\]
This is the intersection of $\mathbb{H}^2$ with a spacelike affine plane whose ``unit'' normal is $p$.
Intersections of $\mathbb{H}^2$ with timelike affine planes are called hypercycles.
A hypercycle can be described by an equation
\[
\langle p, x \rangle_{2,1} = \mathrm{const}, \quad \langle p, p \rangle_{2,1} = 1.
\]
The point $p$ (in the so-called anti de Sitter plane $\{p \in \mathbb{R}^{2,1} \mid \langle p, p \rangle_{2,1} = 1\}$) is called the center of the hypercycle.
Non-empty intersections of $\mathbb{H}^2$ with lightlike affine planes are called horocycles, their centers are lightlike vectors.
Every triangle $p_1p_2p_3$ in $\mathbb{H}^2$ has a unique \emph{circumcycle}, that is a circle or a horocycle or a hypercycle passing through $p_1, p_2, p_3$.
The center $o \in \mathbb{R}^{2,1}$ of the circumcycle satisfies
\[
\langle p_1, o \rangle_{2,1} = \langle p_2, o \rangle_{2,1} = \langle p_3, o \rangle_{2,1}.
\]

Let now $(V, E, F)$ be a triangulated surface without boundary where every triangle is equipped with a hyperbolic metric.
For every edge $ij \in E$ the union of adjacent triangles $ijk$ and $ijl$ can be mapped isometrically to $\mathbb{H}^2$.
Then an analog of Lemma \ref{lem:SpherWeightsPos} holds, with the only difference that the circumcircles are to be replaced with circumcycles.
The centers of circumcycles may be ideal (lightlike vectors) or hyperideal (spacelike vectors), which however has no influence on the computation of angles $\angle(\wideparen{p_ip_j}, \wideparen{p_io_{ijk}})$ and the like.

Similarly to the spherical case one constructs the dual polyhedron in the Minkowski space circumscribed about the hyperboloid.
An analog of Lemma \ref{lem:WeightsGeom} holds, with the only difference that in \eqref{eqn:DualFaceSpher} a sign change in one of the terms occurs:
\[
\sum_j c_{ij}p_j = -\left( 2d_i + \sum_j c_{ij} \right) p_i.
\]

\subsection{Infinitesimal isometric deformations of polyhedra inscribed in a two-sheeted hyperboloid}
Similarly to Section \ref{sec:InfDef}, an infinitesimal isometric deformation of a realization $p \colon V \to \mathbb{R}^{2,1}$ of a graph $(V, E)$ in the Minkowski space is a map $q \colon V \to \mathbb{R}^{2,1}$ such that for every edge $ij$ one has
\[
\langle p_j - p_i, q_j - q_i \rangle_{2,1} = 0,
\]
which is equivalent to
\[
\left. \frac{d}{dt} \right|_{t=0} \langle p_j(t) - p_i(t), p_j(t) - p_i(t) \rangle_{2,1} = 0
\]
for any one-parameter family of realizations $p(t)$ such that $p(0) = p$ and $p'(0) = q$.
Trivial infinitesimal isometric deformations, rotation and translation fields are defined similarly to the Euclidean case, and an analog of Lemma \ref{lem:RotTranslProp} holds.
This allows to apply the arguments from Section \ref{sec:InfDef} to prove the first two parts of the following Theorem.

\begin{theorem}
\label{thm:HypIIDDiscretePrime}
Let $P$ be a triangulated Euclidean polyhedral surface with vertices indexed by the set $V = \{i, j, \ldots\}$, all vertices $p_i$ lying on the upper half of the one-sheeted hyperboloid viewed as a ``sphere'' in the Minkowski space:
\[
\mathbb{H}^2 = \{x \in \mathbb{R}^{2,1} \mid \langle x, x \rangle_{2,1} = -1\}
\]
and such that the projection of $P$ to $\mathbb{H}^2$ from the origin is injective.
Let $q \colon V \to \mathbb{R}^{2,1}$ be an infinitesimal isometric deformation of $P$ with respect to the Minkowski norm squared.
Decompose $q$ into radial and tangent components:
\[
q_i = f_i p_i + \xi_i, \quad f \colon V \to \mathbb{R}, \quad \xi \colon V \to \mathbb{R}^{2,1}, \quad \langle \xi_i, p_i \rangle_{2,1} = 0.
\]
and denote by $P_h$ the triangulated hyperbolic surface obtained by centrally projecting the edges of $P$.
Then the following holds:
\begin{enumerate}
\item
The radial component is an eigenfunction of the discrete hyperbolic Laplacian on $P_h$ with the eigenvalue $2$:
\[
\triangle_h f = 2f.
\]
\item
Conversely, if $P$ is simply connected, then for every $2$-eigenfunction $f$ of the discrete hyperbolic Laplacian on $P_h$ there is an infinitesimal isometric deformation of $P$ with the radial component $f$.
\item
The vector field $\xi$ is an infinitesimal conformal deformation of $P_h$ with the conformal factor $-f$.
\end{enumerate}
\end{theorem}

There is a simple trick that relates infinitesimal isometric deformations of a surface in $\mathbb{R}^{2,1}$ to infinitesimal isometric deformations of the same surface in $\mathbb{R}^3$.

\begin{lemma}
\label{lem:PogMap}
Let $p \colon V \to \mathbb{R}^3$ be a realization of a graph $(V, E)$, and let $q \colon V \to \mathbb{R}^3$ be a vector field on its vertices.
Put $\overline{q} = c \circ q$, where
\[
c \colon \mathbb{R}^3 \to \mathbb{R}^3, \quad c(x_0, x_1, x_2) = (-x_0, x_1, x_2).
\]
Then $q$ is an infinitesimal isometric deformation of $p$ with respect to the Euclidean metric if and only if $\overline{q}$ is an infinitesimal isometric deformation of $p$ with respect to the Minkowski norm squared.
\end{lemma}
\begin{proof}
This is straightforward from
\[
\langle p_j - p_i, q_j - q_i \rangle = \langle p_j - p_i, \overline{q}_j - \overline{q}_i \rangle_{2,1}.
\]
\end{proof}

This was noted in \cite{Galeeva82} and is the simplest instance of a more general construction called infinitesimal Pogorelov maps, see \cite{Fillastre2019}.

Via Lemma \ref{lem:PogMap}, we relate infinitesimal isometric deformations with respect to the Minkowski metric in Theorem \ref{thm:HypIIDDiscretePrime} into those with respect to the Euclidean metric in Theorem \ref{thm:HypIIDDiscrete}.

\begin{proof}[Proof of the equivalence Theorem \ref{thm:HypIIDDiscretePrime} $\Leftrightarrow$ Theorem \ref{thm:HypIIDDiscrete}]
	
Lemma \ref{lem:PogMap} implies that an infinitesimal deformation $q$ is isometric with respect to the Euclidean metric if and only if $\overline{q}$ is isometric with respect to the Minkowski metric. It remains to show that the radial and the tangent components have the required forms.

We write  $\overline{q}_i = -f_i p_i + \xi_i$ with $\langle \xi_i, p_i \rangle_{2,1} = 0$. It yields
\[
\langle p_i, \overline{q}_i \rangle_{2,1} = -f_i \langle p_i, p_i \rangle_{2,1} = f_i \quad \Rightarrow \quad f_i =  \langle p_i, \overline{q}_i \rangle_{2,1} =  \langle p_i, q_i \rangle
\]
and hence
\[
\xi_i = \overline{q}_i + f_i p_i = \overline{q}_i + \langle p_i, q_i \rangle p_i.
\]	
is tangential to the hyperboloid. Thus, we obtain the equivalence of Theorems \ref{thm:HypIIDDiscrete} and \ref{thm:HypIIDDiscretePrime}.

\end{proof}

\subsection{Infinitesimal conformal deformations}
Similarly to Section \ref{sec:InfConfDef}, two triangulated hyperbolic surfaces with the same combinatorics and the edge lengths $\lambda_{ij}$, respectively $\widetilde{\lambda}_{ij}$, are called conformally equivalent if there is a function $u \colon V \to \mathbb{R}$ such that
\[
\sinh \frac{\widetilde{\lambda}_{ij}}2 = e^{\frac{u_i+u_j}2} \sinh \frac{\lambda_{ij}}2
\]
for all edges $ij \in E$.
An infinitesimal conformal deformation of a surface with edge lengths $\lambda_{ij}$ is a map $\dot\lambda \colon E \to \mathbb{R}$ such that
\[
\left( \log\sinh\frac{\lambda_{ij}}2 \right)^{\boldsymbol{\cdot}} = \frac{u_i + u_j}2 \quad \text{or, equivalently} \quad
\dot\lambda_{ij} = (u_i + u_j) \tanh \frac{\lambda_{ij}}2
\]
for some function $u$.
A discrete conformal vector field is a tangent vector field defined on the vertices of a triangulated surface such that the corresponding edge length derivatives satisfy the above condition.

An analog of Lemma \ref{lem:DiscConf} holds and can be proved along the same lines.

\begin{lemma}
A tangent vector field $\xi$ on the vertices of a triangulated hyperbolic surface, locally isometrically mapped to $\mathbb{H}^2$, is discretely conformal if and only if there is a function $u \colon V \to \mathbb{R}$ such that for every edge $ij$ one has
\[
\langle \xi_j - \xi_i, p_j - p_i \rangle_{2,1} = \frac{u_i + u_j}2 \langle p_j - p_i, p_j - p_i \rangle_{2,1}.
\]
\end{lemma}

\begin{proof}[Proof of Theorem \ref{thm:ConfFieldDiscreteHyp}, first part and of Theorem \ref{thm:HypIIDDiscrete}, third part.]
The same arguments as in the spherical case apply.
\end{proof}

The second part of Theorem \ref{thm:ConfFieldDiscreteHyp}, that is the equivalence between two discrete conformality theorems on the infinitesimal level, is also proved similarly to the Euclidean case.
The stereographic projection maps the upper half of the hyperboloid to the unit disk in the plane $x_0 = 1$, resulting in the Poincar\'e disk model of the hyperbolic plane.
The cross-ratio is invariant because the isometries of the Poincar\'e disk model are M\"{o}bius transformations.

\subsection{Circle patterns on closed hyperbolic surfaces} \label{sec:circleinfrigid}

Delaunay triangulations are defined not only on hyperbolic surfaces, but also on surfaces with complex projective structures. One can consider deformations of Delaunay triangulations with varying complex projective structures (See \cite{KMT2006,Lam2021}). The space of complex projective structures contains a sub-manifold consisting of hyperbolic structures, which is called the Fuchsian slice.  With the help of the hyperbolic Laplacian one can deduce the infinitesimal rigidity of circle patterns as well as the transversality between the deformation space of circle patterns and the Fuchsian slice. This is contained implicitly in \cite{Bobenko2010} and can be deduced from \cite{Schlenker2007}. The infinitesimal rigidity is crucial to the proof of the smoothness of the deformation space of circle patterns  \cite{BW2023,lam2024sym}.

\begin{theorem}\label{thm:disrigid}
	Given a Delaunay geodesic triangulation of a closed hyperbolic surface, every infinitesimal conformal deformation of a circle pattern staying within the hyperbolic slice must be trivial. 
\end{theorem}

\begin{proof}
	Suppose there is an infinitesimal deformation of a circle pattern staying on the slice of Fuchsian structures and preserving the length cross-ratios.  
	Then the change of hyperbolic edge length $\lambda$ satisfies
	\[
	\left( \log\sinh\frac{\lambda_{ij}}2 \right)^{\boldsymbol{\cdot}} = \frac{u_i + u_j}2
	\]
	for some $u:V \to \mathbb{R}$. Theorem \ref{thm:ConfFieldDiscreteHyp} implies that $\triangle_h u = 2u$, and hence
	\[
	\langle \triangle_h u, u \rangle_d = 2 \langle u, u \rangle_d = 2 \sum_i d_i u_i^2 \ge 0.
	\]
	On the other hand, for every $u$ one has
	\[
	\langle \triangle_h u, u \rangle_d = -\sum_{ij \in E} c_{ij} (u_j - u_i)^2 \le 0.
	\]
	It follows that $u\equiv 0$ and the infinitesimal deformation is trivial.
\end{proof}

Analogous statement holds for deformations preserving intersection angles, but in a weaker form.
\begin{corollary}\label{cor:circlerigid}
	Given a Delaunay geodesic triangulation of a closed hyperbolic surface, every infinitesimal deformation of a circle pattern preserving the intersection angles while preserving the hyperbolic structure must be trivial. 
\end{corollary}
\begin{proof}
	Suppose an infinitesimal deformation of a circle pattern preserving the intersection angles and the hyperbolic structure is given. Applying the $90^\circ$-rotation $J$, one obtains an infinitesimal deformation preserving the length cross ratios and the hyperbolic structure, which has to be trivial.
\end{proof}
 
The reason for a weaker statement is that even if an infinitesimal deformation stays within the hyperbolic slice, its $J$-rotation might not stay in the hyperbolic slice anymore. This can be deduced from the change in $SL(2,\mathbb{C})$-holonomy representation as cross ratios vary, while the hyperbolic structures have holonomy representations in $SL(2,\mathbb{R})$ (See \cite[Proposition 4.5]{lam2024sym}).

\bibliographystyle{amsplain}
\bibliography{lap}

\end{document}